\documentclass[10pt]{article}
\usepackage[utf8]{inputenc} % allow utf-8 input
\usepackage[T1]{fontenc}    % use 8-bit T1 fonts
\usepackage{hyperref}       % hyperlinks
\usepackage{url}            % simple URL typesetting
\usepackage{booktabs}       % professional-quality tables
\usepackage{amsfonts}       % blackboard math symbols
\usepackage{nicefrac}       % compact symbols for 1/2, etc.

\usepackage{amsmath,amssymb,amsthm,nccmath,mathtools}

\usepackage[ruled]{algorithm2e}
\usepackage{enumitem}
\usepackage{multirow}
\usepackage{array}
\usepackage{graphicx,subcaption}

\usepackage{geometry,comment}
\usepackage[noblocks]{authblk}

\usepackage{xcolor}
\usepackage{pbox,adjustbox}

\makeatletter
\newcommand{\removelatexerror}{\let\@latex@error\@gobble}
\makeatother

\setlist{leftmargin=0pt,itemindent=15pt,labelwidth=8pt,labelsep=7pt,listparindent=15pt,topsep=2pt, itemsep = 2pt}

\title{\bf Distributed Zero-Order Algorithms for Nonconvex Multi-Agent Optimization%
}
\date{}

\author[1]{Yujie Tang}
\author[2]{Junshan Zhang}
\author[1]{Na Li}
\affil[1]{School of Engineering and Applied Sciences,
Harvard University}
\affil[2]{School of Electrical, Computer and Energy Engineering, Arizona State University}

\newtheorem{theorem}{Theorem}
\newtheorem{proposition}{Proposition}
\newtheorem{lemma}{Lemma}
\newtheorem{corollary}{Corollary}

\newtheorem{assumption}{Assumption}

\theoremstyle{definition}
\newtheorem{definition}{Definition}

\theoremstyle{remark}
\newtheorem{remark}{Remark}

\begin{document}

\maketitle

\begin{abstract}
Distributed multi-agent optimization finds many applications in distributed learning, control, estimation, etc. Most existing algorithms assume knowledge of first-order information of the objective and have been analyzed for convex problems. However, there are situations where the objective is nonconvex, and one can only evaluate the function values at finitely many points. In this paper we consider derivative-free distributed algorithms for nonconvex multi-agent optimization, based on recent progress in zero-order optimization. We develop two algorithms for different settings, provide detailed analysis of their convergence behavior, and compare them with existing centralized zero-order algorithms and gradient-based distributed algorithms.

\vspace{5pt}
\end{abstract}

\section{Introduction}

Consider a set of $n$ agents connected over a network, each of which is associated with a smooth local objective function $f_i$ that can be nonconvex. The goal is to solve the optimization problem
$$
\min_{x\in\mathbb{R}^d}\ \ 
f(x)\coloneqq 
\frac{1}{n}\sum_{i=1}^n f_i(x)
$$
with the restriction that $f_i$ is only known to agent $i$ and each agent can exchange information only with its neighbors in the network during the optimization procedure. We focus on the situation where only zero-order information of $f_i$ is available to agent $i$.

\begin{table}[!t]
\aboverulesep = 1.2mm
\belowrulesep = 1.2mm
\renewcommand{\arraystretch}{1.25}
  \caption{Comparison of different algorithms for distributed optimization and zero-order optimization.}
  \label{tab:main_results}
    \centering\small
\begin{adjustbox}{width=\columnwidth,center}
    \begin{tabular}{>{\centering\arraybackslash}m{1.8 cm}>{\centering\arraybackslash}m{2.7cm}cc}
    \toprule
    & &  smooth & gradient dominated \\
    \midrule
    \multirow{6}{1.8cm}{\centering distributed zero-order (nonconvex)}
    & Alg. \ref{alg:2point}, this paper ($2$-point + DGD)
    & $\quad O\!\left(\!\sqrt{\dfrac{d}{m}}\log m\!\right)$
    & $O\!\left(\dfrac{d}{m}\right)$ \\
    \cmidrule{2-4}
    & Alg. \ref{alg:multipoint}, this paper ($2d$-point + gradient tracking)
    & $O\!\left(\dfrac{d}{m}\right)$
    & $O\!\left(\!\left[1-c\big(1-\rho^2\big)^2\left(\dfrac{\mu}{L}\right)^{\frac{4}{3}}\right]^{m/d}\right)$ 
    \\
    \cmidrule{2-4}
    & ZONE \cite{hajinezhad2019zone}
    & $O\left(\dfrac{\gamma(d)}{\sqrt{M}}\right)$
    & ---
    \\
    \midrule
    \multirow{5}{1.8cm}{\centering distributed first-order}
     & \multirow{3}{*}{\centering DGD} & $O\!\left(\dfrac{\log t}{\sqrt{t}}\right)$ \cite{chen2012fast,zeng2018nonconvex} (convex)
     & \multirow{3}{*}{\centering $O\!\left(\dfrac{ 1}{t}\right)$ \cite{olshevsky2019non} (strongly convex)} \\[8pt]
     & & $O\!\left(\!\dfrac{1}{\sqrt{T}}\!\right)$ \cite{lian2017can} (nonconvex)
     \\
    \cmidrule{2-4}
    & \pbox{2.3cm}{\relax\ifvmode\centering\fi gradient\\ tracking} &
    $O\!\left(\dfrac{1}{t}\right)$ \cite{scutari2019distributed} (nonconvex)&
    $O\!\left(\!\left[1 \!-\! c(1 \!-\! \rho)^2\!\left(\dfrac{\mu}{L}\right)^{\!\frac{3}{2}}\right]^{\!t}\right)$
    \cite{nedic2017achieving} (strongly convex) \\
    \midrule
    centralized zero-order& 
    \pbox{2.7cm}{\relax\ifvmode\centering\fi
    \cite{nesterov2017random} \\ ($2$-point estimator)}
    & $O\!\left(\dfrac{d}{m}\right)$ (nonconvex)
    & $O\!\left(\!\left[1\!-\!\dfrac{c}{d}\dfrac{\mu}{L}\right]^m\right)$ (strongly convex)\\
    \bottomrule
    \\[-8pt]
    \multicolumn{4}{l}{\scriptsize Note: The table summarizes best known convergence rates for deterministic nonconvex unconstrained optimization with 1) smooth, } \\[-6pt]
    \multicolumn{4}{l}{\scriptsize
    \phantom{Note: }2) gradient dominated objectives. The convex counterparts are listed if results for nonconvex cases have not been established.} \\[-4pt]
    \multicolumn{4}{l}{\scriptsize\phantom{Note: }$m$ denotes the number of function value queries, $t$ denotes the number of iterations, $d$ denotes the dimension of the decision} \\[-6pt]
    \multicolumn{4}{l}{\scriptsize\phantom{Note: }variable, $c$'s represent numerical constants that can be different for different algorithms.} \\[-4pt]
    \multicolumn{4}{l}{\scriptsize\phantom{Note: }$M$ denotes the total number of function value queries and $T$ denotes the total number of iterations provided before the}  \\[-6pt]
    \multicolumn{4}{l}{\scriptsize\phantom{Note: }optimization procedure. The rates in \cite{hajinezhad2019zone} and \cite{lian2017can} assume constant step sizes chosen based on $M$ or $T$.} \\[-4pt]
    \multicolumn{4}{l}{\scriptsize\phantom{Note: }The listed convergence rates are the ergodic rates of $\|\nabla f\|^2$ for the smooth case, and the objective error rates for the} \\[-4pt]
    \multicolumn{4}{l}{\scriptsize\phantom{Note: }gradient dominated case, respectively.} \\[-4pt]
    \multicolumn{4}{l}{\scriptsize\phantom{Note: }The rates provided in \cite{hajinezhad2019zone} do not include explicit dependence on $d$; we use $\gamma(d)$ to denote this dependence.} \\[-4pt]
    \multicolumn{4}{l}{\scriptsize\phantom{Note: }The cited results in this table may apply to more general settings (e.g., stochastic gradients \cite{lian2017can,olshevsky2019non}).} \\[-4pt]
    \multicolumn{4}{l}{\scriptsize\phantom{Note: }We do not include algorithms with Nesterov-type acceleration in this comparison.}
    \end{tabular}
\end{adjustbox}
\vspace{-10pt}
\end{table}

Distributed multi-agent optimization lies at the core of a wide range of applications, and a large body of literature has been contributed to distributed multi-agent optimization algorithms. One line of research combines (sub)gradient-based methods with a consensus/averaging scheme, where each iteration of a local agent consists of one or multiple consensus steps and a local gradient evaluation step. It has been shown that, for convex functions, the convergence rates of distributed gradient-based algorithms can match or nearly match those of centralized gradient-based algorithms. Specifically, \cite{nedic2009distributed,chen2012fast} proposed and analyzed consensus-based decentralized gradient descent (DGD) algorithms with $O(\log t/\sqrt{t})$ convergence for nonsmooth convex functions; \cite{shi2015extra,nedic2017achieving,qu2018harnessing} employed the \textit{gradient tracking} scheme and showed that the DGD with gradient tracking achieves $O(1/t)$ convergence for smooth convex functions and linear convergence for strongly convex functions; \cite{qu2019accelerated} employed Nesterov's gradient descent method and showed $O(1/t^{1.4-\epsilon})$ convergence for smooth convex functions and improved linear convergence for strongly convex functions where $\epsilon$ is an arbitrarily small positive number. Besides convergence rates, some works have additional focuses such as time-varying/directed graphs \cite{nedic2014distributed}, uncoordinated step sizes \cite{xu2015augmented}, stochastic (sub)gradient \cite{ram2010distributed}, etc.

While distributed convex optimization has broad applicability, nonconvex problems also appear in important applications such as distributed learning \cite{omidshafiei2017deep}, robotic networks \cite{charrow2014cooperative}, operation of wind farms \cite{marden2013model}, etc. Several works have considered nonconvex multi-agent optimization and developed various distributed gradient-based methods to converge to stationary points with convergence rate analysis, e.g., \cite{di2016next,lian2017can,zeng2018nonconvex,scutari2019distributed}. We notice that for smooth functions, either convex or nonconvex, in general DGD with \textit{gradient-tracking} converges faster than the method without gradient tracking, and its convergence rate has the same big-O dependence on the number of iterations as the centralized vanilla gradient descent method (See Table \ref{tab:main_results}).

Further, there has been increasing interest in zero-order optimization, where one does not have access to the gradient of the objective. Such situations can occur, for example, when only black-box procedures are available for computing the values of the functional characteristics of the problem, or when resource limitations restrict the use of fast or automatic differentiation techniques. Many existing works \cite{kiefer1952stochastic,flaxman2005online,bach2016highly,nesterov2017random,duchi2015optimal} on zero-order optimization are based on constructing gradient estimators using finitely many function evaluations, e.g., gradient estimator based on Kiefer-Wolfowitz scheme\cite{kiefer1952stochastic} by using $2d$-point function evaluations where $d$ is the dimension of the problem. However, this estimator does not scale up well with high-dimensional problems.
\cite{flaxman2005online} proposed and analyzed a single-point gradient estimator, and \cite{bach2016highly} further studied the convergence rate for highly smooth objectives. \cite{nesterov2017random} proposed two-point gradient estimators and showed that the convergence rates of the resulting algorithms are comparable to their first-order counterpart (See Table I). For instance, gradient descent with two-point gradient estimators converges with a rate of $O(d/m)$ where $m$ denotes the number of function value queries. \cite{duchi2015optimal} and \cite{shamir2017optimal} showed that two-point gradient estimators achieve the optimal rate $O(\sqrt{d/m})$ of stochastic zero-order convex optimization.

Some recent works have started to combine zero-order and distributed optimization methods \cite{hajinezhad2019zone,sahu2018distributed,yu2019distributed}. For example, \cite{hajinezhad2019zone} proposed the ZONE algorithm for stochastic nonconvex problems based on the method of multipliers. \cite{sahu2018distributed} proposed a distributed zero-order algorithm over random networks and established its convergence for strongly convex objectives. \cite{yu2019distributed} considered distributed zero-order methods for constrained convex optimization. However, there are still many questions remaining to be studied in distributed zero-order optimization. In particular, \textit{how do zero-order and distributed methods affect the performance of each other}, and \textit{could their fundamental structural properties be kept when combining the two?} For instance, it would be ideal if we could combine both $2$-point zero-order methods with DGD with gradient tracking and maintain the nice properties for both methods, leading to an ``optimal'' distributed zero-order algorithm if possible. This is unclear {\em a prior}, and indeed, as we shall show later, $2$-point gradient estimator and DGD with gradient tracking do not reconcile with each other well. 

\paragraph{Contributions.} Motivated by the above observations, we propose two distributed zero-order algorithms: Algorithm 1 is based on the $2$-point estimator and DGD; Algorithm 2 is based on the $2d$-point gradient estimator and DGD with gradient tracking. We analyze the performance of the two algorithms for deterministic nonconvex optimization, and compare their convergence rates with their distributed first-order and centralized zero-order counterparts. The convergence rates of the two algorithms are summarized in Table \ref{tab:main_results}. Specifically, it can be seen that the rates of Algorithm 1 are comparable with the first-order decentralized gradient descent but are inferior to the centralized zero-order method; the rates of Algorithm 2 are comparable with the centralized zero-order method and the first-order DGD with gradient tracking. On the other hand, Algorithm 1 uses the $2$-point gradient estimator that requires only $2$ function value queries, while Algorithm 2 employs the $2d$-gradient estimator whose computation involves $2d$ function value queries, indicating that Algorithm 1 could be favored for high-dimensional problems even though its convergence is slower asymptotically, while Algorithm \ref{alg:multipoint} could handle problems of relatively low dimensions better with faster convergence. These results shed light on how zero-order evaluations affect distributed optimization and how the presence of network structure affects zero-order algorithms. Different problems and different computation requirements would favor different integration of zero-order methods and distributed methods.

Compared to existing literature on distributed zero-order optimization, our Algorithm~1 is similar to the algorithms proposed in \cite{sahu2018distributed,yu2019distributed}, but our analysis assumes nonconvex objectives and also considers gradient dominated functions. While \cite{hajinezhad2019zone} analyzed the performance of the ZONE algorithm for unconstrained nonconvex problems, we shall see that our Algorithm~1 achieves comparable convergence behavior with ZONE-M, and Algorithm~2 converges faster than ZONE-M in the deterministic setting due to the use of the gradient tracking technique. A more detailed comparison will be given in Section~\ref{subsec:compare_other_algs}.

\paragraph*{Notation.}
We denote the $\ell_2$-norm of vectors and matrices by $\|\cdot\|$. The standard basis of $\mathbb{R}^d$ will be denoted by $\{e_k\}_{k=1}^d$. We let $\mathbf{1}_n\in\mathbb{R}^n$ denote the vector of all ones. We let $\mathbb{B}_d$ denote the closed unit ball in $\mathbb{R}^d$, and let $\mathbb{S}_{d-1}\coloneqq \{x\in\mathbb{R}^d:\|x\|=1\}$ denote the unit sphere. The uniform distributions over $\mathbb{B}_d$ and $\mathbb{S}_{d-1}$ will be denoted by $\mathcal{U}(\mathbb{B}_{d})$ and $\mathcal{U}(\mathbb{S}_{d-1})$. $I_d$ denotes the $d\times d$ identity matrix. For two matrices $A=[a_{ij}]\in\mathbb{R}^{p\times q}$ and $B=[b_{ij}]\in\mathbb{R}^{r\times s}$, their tensor product $A\otimes B$ is
$$
A\otimes B
=
\begin{bmatrix}
a_{11} B & \cdots & a_{1q} B \\
\vdots & \ddots &\vdots \\
a_{p1} B & \cdots & a_{pq} B
\end{bmatrix}
\in\mathbb{R}^{pr\times qs}.
$$

\section{Formulation and Algorithms}

\subsection{Problem Formulation}
Let $\mathcal{N} = \{1,2,\ldots,n\}$ be the set of agents. Suppose the agents are connected by a communication network, whose topology is represented by an undirected, connected graph $\mathcal{G}=(\mathcal{N},\mathcal{E})$ where the edges in $\mathcal{E}$ represent communication links.

Each agent $i$ is associated with a local objective function $f_i:\mathbb{R}^d\rightarrow\mathbb{R}$. The goal of the agents is to collaboratively solve the optimization problem
\begin{equation}\label{eq:main_problem}
\min_{x\in\mathbb{R}^d}\ \ 
f(x)\coloneqq 
\frac{1}{n}\sum_{i=1}^n f_i(x).
\end{equation}
We assume that at each time step, agent $i$ can only query the function values of $f_i$ at finitely many points, and can only communicate with its neighbors. Similar to \cite{nesterov2017random} and other works on zero-order optimization, we assume a deterministic setting where the queries of the function values are \emph{noise-free} and \emph{error-free}. The analysis of the deterministic setting will provide a baseline for extension to stochastic optimization which we leave as future work.

The following definitions will be useful later in the paper.
\begin{definition}
\begin{enumerate}
\item A function $f:\mathbb{R}^d\rightarrow\mathbb{R}$ is said to be \emph{$L$-smooth} if $f$ is continuously differentiable and satisfies
$$
\|\nabla f(x)-\nabla f(y)\|\leq L\|x-y\|,
\qquad
\forall x,y\in\mathbb{R}^d.
$$
\item A function $f:\mathbb{R}^d\rightarrow\mathbb{R}$ is said to be \emph{$G$-Lipschitz} if
$$
|f(x)-f(y)|\leq G\|x-y\|
\qquad
\forall x,y\in\mathbb{R}^d.
$$
\item A function $f:\mathbb{R}^d\rightarrow\mathbb{R}$ is said to be \emph{$\mu$-gradient dominated} if $f$ is differentiable, has a global minimizer $x^\ast$, and
$$
2\mu(f(x)-f(x^\ast)) \leq \|\nabla f(x)\|^2
\qquad
\forall x\in\mathbb{R}^d.
$$
\end{enumerate}
\end{definition}

The notion of gradient domination is also known as Polyak-{\L}ojasiewicz (PL) inequality, first introduced by \cite{polyak1963gradient} and \cite{lojasiewicz1963topological}. It can be viewed as a nonconvex analogy of strong convexity, as the centralized vanilla gradient descent achieves linear convergence for gradient dominated objective functions. The gradient domination condition has been frequently discussed in nonconvex optimization \cite{polyak1963gradient,karimi2016linear}. Also, nonconvex but gradient dominated objective functions appear in many applications, e.g., linear quadratic control problems \cite{fazel2018global} and deep linear neural networks \cite{shamir2018exponential}. 

\subsection{Preliminaries on Zero-Order and Distributed Optimization}\label{subsec:prelim}

We present some preliminaries to motivate our algorithm development.

\vspace{6pt}
\noindent \textbf{Zero-order optimization based on gradient estimation.} In zero-order optimization, one tries to minimize a function with the limitation that only function values at finitely many points may be obtained. One basic approach of designing zero-order optimization algorithms is to construct gradient estimators from zero-order information and substitute them for the true gradients. Here we introduce two types of zero-order gradient estimators for the noiseless setting:
\begin{enumerate} 
\item[i)] The \emph{$2d$-point gradient estimator} is given by
\begin{equation}\label{eq:multipoint_est}
\mathsf{G}^{(2d)}_f(x;u)
=
\sum_{k=1}^d\frac{f(x+ue_k)-f(x-ue_k)}{2u}e_k,
\end{equation}
where $u$ is some given positive number. Basically, it approximates the gradient $\nabla f(x)$ by taking finite differences along $d$ orthogonal directions, and can be viewed as a noise-free version of the classical Kiefer-Wolfowitz type method \cite{kiefer1952stochastic}. Given an $L$-smooth function $f:\mathbb{R}^d\rightarrow\mathbb{R}$, it can be shown that
\begin{equation*}
\|\mathsf{G}^{(2d)}_f(x;u)
-\nabla f(x)\|\leq \mfrac{1}{2}uL\sqrt{d}
\end{equation*}
for any $x\in\mathbb{R}^d$. The right-hand side decreases to zero as $u\rightarrow 0$. In other words, $\mathsf{G}^{(2d)}_f(x;u)$ can be arbitrarily close to $\nabla f(x)$ (as long as the finite differences can be evaluated accurately). One drawback of this estimator is that it requires $2d$ zero-order queries, which may not be computationally efficient for high-dimensional problems.

\item[ii)] The \textit{$2$-point gradient estimator} is given by
\begin{equation}\label{eq:2point_est}
\mathsf{G}^{(2)}_f(x;u,z)
\coloneqq 
d\cdot \frac{f(x+uz)-f(x-uz)}{2u}z,
\end{equation}
where $z\in\mathbb{R}^d$ is a random vector that is sampled from the distribution $\mathcal{U}(\mathbb{S}_{d-1})$, and $u>0$ is a given positive number. The following proposition indicates that when $z$ is uniformly sampled from the sphere $\mathbb{S}_{d-1}$, the expectation of $\mathsf{G}^{(2)}_f(x;u,z)$ is the gradient of a ``locally averaged'' version of $f$.
\begin{proposition}[{\hspace{1sp}\cite{flaxman2005online}}]\label{prop:basic_lemma_2point}
Suppose $f:\mathbb{R}^d\rightarrow\mathbb{R}$ is $L$-smooth. Then for any $u>0$ and $x\in\mathbb{R}^d$,
$$
\mathbb{E}_{z\sim\mathcal{U}(\mathbb{S}_{d-1})}
\!\left[\mathsf{G}^{(2)}_f(x;u,z)\right]
=\nabla f^u(x),
$$
where $f^u(x)\coloneqq \mathbb{E}_{y\sim\mathcal{U}(\mathbb{B}_d)}
\left[f(x+uy)\right]$.
\end{proposition}
%On the other hand, it can be shown that the variance of $\mathsf{G}^{(2)}_f(x;u,z)$ goes to $(d-1)\|\nabla f(x)\|^2$ as $u\rightarrow 0$ for smooth function $f$ (see Appendix \ref{sec:app:variance}). This suggests that the variance of $\mathsf{G}^{(2)}_f(x;u,z)$ in general cannot be made arbitrarily small (unless $d=1$ or we are approaching a stationary point of $f$). Nevertheless,
It has been shown in \cite{nesterov2017random} that if we substitute $\mathsf{G}^{(2)}_f(x;u,z)$ for the gradient in the gradient descent algorithm, we have
$$
\frac{1}{t}\sum_{\tau=0}^{t-1}\|\nabla f(x_\tau)\|^2
=O\!\left(\frac{d}{m}\right)
$$
for nonconvex smooth objectives, and
$$
f(x_\tau)-f^\ast = O\!\left(\!\left[1-c\,\frac{\mu/L}{d}\right]^{\!m}\right)
$$
for smooth and strongly convex objectives, where $x_\tau$ denotes the $\tau$'th iterate and $m$ denotes the number of zero-order queries in $t$ iterations (see Table \ref{tab:main_results}). These rates are comparable to the rates of the (centralized) vanilla gradient descent method, i.e., $O(1/t)$ for nonconvex smooth objectives and linear convergence for smooth and strongly convex objectives.
\end{enumerate}

\vspace{6pt}
\noindent \textbf{Distributed optimization.} In this paper, we mainly focus on consensus-based algorithms for distributed optimization, where each agent maintains a local copy of the global variables, and weighs its neighbors’ information to updates its own local variable. Specifically, for a time-invariant and bidirectional communication network, we introduce 
a consensus matrix $W=[W_{ij}]\in\mathbb{R}^{n\times n}$ that satisfies the following assumption:
\begin{assumption}\label{assumption:consensus_weight}
\begin{enumerate}
\item $W$ is a doubly stochastic matrix.
\item $W_{ii}>0$ for all $i\in\mathcal{N}$, and for two distinct agents $i$ and $j$, $W_{ij}> 0$ if and only if $(i,j)\in\mathcal{E}$.
\end{enumerate}
\end{assumption}
When Assumption \ref{assumption:consensus_weight} is satisfied, we have \cite{qu2018harnessing}
\begin{equation}\label{eq:def_rho_net}
\rho \coloneqq 
\left\|W-n^{-1}\mathbf{1}_n\mathbf{1}_n^\top\right\|<1.
\end{equation} 
We present two consensus-based algorithms that will serve as the basis for designing distributed zero-order algorithms.
\begin{enumerate}
\item[i)] The \emph{decentralized gradient descent (DGD)} algorithm \cite{nedic2009distributed,chen2012fast} is given by the following iterations:
\begin{equation}\label{eq:DGD}
x^i(t)=\sum_{j=1}^n W_{ij}x^j(t-1)-\eta_t\nabla f_i(x^i(t-1)),
\end{equation}
where $x^i(t)\in\mathbb{R}^d$ denotes the local copy of the decision variable for the $i$'th agent, and $\eta_t$ is the step size. It has been shown that DGD in general converges more slowly than the centralized gradient descent algorithm \cite{chen2012fast,qu2018harnessing} for smooth functions. This is because the local gradient $\nabla f_i$ does not vanish at the stationary point, and a diminishing step size $\eta_t$ is necessary, which slows down the convergence.

\vspace{3pt}
\item[ii)] The \emph{DGD gradient tracking} method incorporates additional local variables $s^i(t)$ to track the global gradient $\nabla f=\frac{1}{n}\sum_{i}\nabla f_i$:
$$
\begin{aligned}
s^i(t) &=\sum_{j=1}^n W_{ij}s^j(t\!-\!1) + \nabla f_i(x^i(t\!-\!1))-\nabla f_i(x^i(t\!-\!2)), \\
x^i(t) &=\sum_{j=1}^n W_{ij}x^j(t\!-\!1)-\eta_t s^i(t),
\end{aligned}
$$
where we set $s^i(0)=\nabla f_i(x^i(0))$ for each $i$. Since gradient tracking has been proposed, it has attracted much attention and inspired many recent studies \cite{xu2015augmented, di2016next,nedic2017achieving,qu2018harnessing,scutari2019distributed}, as it can accelerate the convergence for smooth objectives compared to DGD. Here we provide a high level explanation of how gradient tracking works: For smooth functions, when $x^i(t)$ approaches consensus,  $\nabla f_i(x^i(t))$ will not change much because of the smoothness, and therefore the local variables $s^i(t)$ will eventually reach a consensus; on the other hand, by induction it can be shown that
$$
\frac{1}{n}\sum_{i=1}^n s^i(t)
=\frac{1}{n}\sum_{i=1}^n \nabla f_i(x^i(t)).
$$
Therefore, the sequence $(s^i(t))_{t\in\mathbb{N}}$ will eventually converge to the global gradient, and a constant stepsize $\eta_t=\eta$ is allowed, leading to comparable convergence rates as the centralized gradient methods.  See \cite[Section III and Section IV.B]{qu2018harnessing} for more discussion.
\end{enumerate}

\subsection{Our Algorithms}

Following the previous discussions, it would be ideal if we can combine the $2$-point gradient estimator and the DGD with gradient tracking and maintain a convergence rate comparable to the centralized vanilla gradient descent method. However, it turns out that such combination does not lead to the desired convergence rate. This is mainly because gradient tracking requires increasingly accurate local gradient information as one approaches the stationary point to achieve faster convergence compared to DGD, whereas the $2$-point gradient estimator can produce a variance that does not decrease to zero even if the radius $u$ decreases to zero; a more detailed explanation will be provided in Section \ref{subsec:compare}.  

We propose the following two distributed zero-order algorithms for the problem \eqref{eq:main_problem}:\footnote{
For both algorithms we employ the \emph{adapt-then-combine} (ATC) strategy \cite{sayed2014diffusion}, a commonly used variant for consensus optimization which is slightly different from the combine-then-adapt (CTA) strategy in \eqref{eq:DGD}. Both ATC and CTA can be used in our algorithms, and the convergence results will be similar.}

\begin{figure}[h]
\vspace{-6pt}
{
\removelatexerror
\begin{algorithm}[H]
\SetAlgoNoLine
\DontPrintSemicolon
\For{$t= 1,2,3,\ldots$}{
\ForEach{$i\in\mathcal{N}$}{
\begin{enumerate}[topsep=2pt,itemsep=5pt,rightmargin=45pt]
\item Generate $z^i(t)\sim\mathcal{U}(\mathbb{S}_{d-1})$ independently from $(z^i(\tau))_{\tau=1}^{t-1}$ and $(z^j(\tau))_{\tau=1}^t$ for $j\neq i$.\;
\item Update $x^i(t)$ by
\begin{align}
g^i(t) &=
\mathsf{G}_{f_i}^{(2)}(x^i(t-1); u_t,z^i(t)), \\
x^i(t) &=
\sum_{j=1}^n W_{ij}(x^j(t-1)
-\eta_t g^j(t)).
\label{eq:alg1_iteration}
\end{align}
\vspace{-12pt}
\end{enumerate}
}
}
\caption{2-point gradient estimator without global gradient tracking \label{alg:2point}}
\end{algorithm}
\vspace{4pt}
\begin{algorithm}[H]
\SetAlgoNoLine
\DontPrintSemicolon
Set $s^i(0)=g^i(0)=0$ for each $i\in\mathcal{N}$. \\
\For{$t= 1,2,3,\ldots$}{
\ForEach{$i\in\mathcal{N}$}{
\begin{enumerate}[topsep=2pt,itemsep=5pt,rightmargin=45pt]
\item Update $s^i(t)$ by
\begin{align}
g^i(t)
& =
\mathsf{G}^{(2d)}_{f_i}
(x^i(t-1);u_t), \\
s^i(t) &=\!\sum_{j=1}^n W_{ij} \!\left( s^j(t\!-\!1) \!+\! g^j(t) \!-\! g^j(t\!-\!1)\right)\!.
\end{align}

\item Update $x^i(t)$ by
\vspace{-5pt}
\begin{equation}
x^i(t) =
\sum_{j=1}^n W_{ij}(x^j(t-1)
-\eta s^j(t)).
\phantom{XXX}
\end{equation}
\vspace{-8pt}
\end{enumerate}
}
}
\caption{$2d$-point gradient estimator with global gradient tracking \label{alg:multipoint}}
\end{algorithm}
}
\vspace{-4pt}
\end{figure}

\begin{enumerate}
\item Algorithm \ref{alg:2point} employs the $2$-point gradient estimator \eqref{eq:2point_est}, and adopts the consensus procedure of the DGD algorithm that only involves averaging over the local decision variables.
\item Algorithm \ref{alg:multipoint} employs the $2d$-point gradient estimator \eqref{eq:multipoint_est}, and adopts the consensus procedure of the gradient tracking method where the auxiliary variable $s^i(t)$ is introduced to track the global gradient $\nabla f=\frac{1}{n}\sum_i\nabla f_i$. We shall see in Theorems \ref{theorem:alg2_non_grad_dom} and \ref{theorem:alg2_grad_dom} that $s^i(t)$ converges to the gradient of the global objective function as $t\rightarrow\infty$ under mild conditions.
\end{enumerate}

\section{Main Results}

In this section we present the convergence results of our algorithms. The proofs are postponed to the Appendix.

\subsection{Convergence of Algorithm \ref{alg:2point}}

Let $x^i(t)$ denote the sequence generated by Algorithm \ref{alg:2point} with a positive, non-increasing sequence of step sizes $\eta_t$. Denote
$$
\bar x(t)\coloneqq \frac{1}{n}\sum_{i=1}^n x^i(t),
\qquad
R_0\coloneqq \frac{1}{n}\sum_{i=1}^n\|x^i(0)-\bar x(0)\|^2.
$$

We first analyze the case with general nonconvex smooth objective functions.
\begin{theorem}\label{theorem:alg1_non_grad_dom}
Assume that each local objective function $f_i$ is uniformly $G$-Lipschitz and $L$-smooth for some positive constants $G$ and $L$, and that
$f^\ast\coloneqq \inf_{x\in\mathbb{R}^d}f(x)>-\infty$.
\begin{enumerate}
\item Suppose $\eta_1 L\leq 1/4$, $\sum_{t=1}^\infty \eta_t=+\infty$, $\sum_{t=1}^\infty \eta_t^2<+\infty$, and $\sum_{t=1}^\infty \eta_t u_t^2<+\infty$. Then almost surely, $\|x^i(t)-\bar x(t)\|$ converges to zero for all $i\in\mathcal{N}$, $\nabla f(\bar x(t))$ converges to zero, and $\lim_{t\rightarrow\infty}f(\bar x(t))$ exists.

% \item Suppose now that
% \begin{equation*}
% \eta_t=\frac{\alpha_\eta }{4L\sqrt{d}}\cdot\frac{1}{t^\beta},
% \qquad
% u_t\leq \frac{\alpha_u G}{L\sqrt{d}}
% \cdot\frac{1}{t^{(\gamma-\beta)/2}}
% \end{equation*}
% with $\alpha_\eta\in(0,1]$, $\alpha_u\geq 0$, $\beta\in(1/2,1)$ and $\gamma> 1$. Then
% \begin{equation}\label{eq:alg1_bound_grad_1}
% \begin{aligned}
% &\frac{\sum_{\tau=0}^{t-1} \!
% \eta_{\tau+1} \mathbb{E} \!\left[\|\frac{1}{G}\nabla f(\bar x(\tau))\|^2\right]}
% {\sum_{\tau=0}^{t-1}
% \eta_{\tau+1}} \\
% \leq \ &
% \frac{(1 \!\!-\!\!\beta)\sqrt{d}}{t^{1-\beta}}
% \Bigg[
% \frac{16(f(\bar x(0)) \!\!- \!\! f^\ast)}{\alpha_{\eta} G^2/L}
% +
% \frac{12 R_0L^2/G^2}{n(1\!-\!\rho^2)\sqrt{d}} \\
% &
% +\frac{4\alpha_\eta}{(6\beta \!- \! 3)n^2}
% \!+\!
% \frac{18\alpha_\eta^2 \kappa^2\rho^2}{\sqrt{d}(1 \!- \! \rho^2)^2}
% \!+\!
% \frac{9\alpha_u^2\gamma}{2(\gamma \!-\! 1)}
% \Bigg]
% +o\!\left(\!\frac{1}{t^{1\!-\!\beta}}\!\right),
% \end{aligned}
% \end{equation}
% and
% \begin{equation}\label{eq:alg1_consensus_1}
% \frac{1}{n}\sum_{i=1}^n\mathbb{E}\left[\|x^i(t)-\bar x(t)\|^2\right]
% \leq \frac{\alpha_\eta^2\kappa^2\rho^2}{(1-\rho^2)^2}\frac{ G^2/L^2}{t^{2\beta}}+o(t^{-2\beta}),
% \end{equation}

\item Suppose that
\begin{equation*}
\eta_t=\frac{\alpha_\eta }{4L\sqrt{d}}\cdot\frac{1}{\sqrt{t}},\qquad u_t\leq \frac{\alpha_u G}{L\sqrt{d}}\cdot \frac{1}{t^{\gamma/2-1/4}}
\end{equation*}
with $\alpha_\eta\in(0,1]$, $\alpha_u\geq 0$ and $\gamma>1$. Then almost surely, $\|x^i(t)-\bar x(t)\|$ converges to zero for all $i$, and $
\liminf_{t\rightarrow\infty}\|\nabla f(\bar x(t))\|=0$.
Furthermore, we have
\begin{equation}\label{eq:alg1_bound_grad_2}
\begin{aligned}
\frac{\sum_{\tau=0}^{t-1}\!
\eta_{\tau+1} \mathbb{E}\!\left[\|\nabla f(\bar x(\tau))\|^2\right]}
{\sum_{\tau=0}^{t-1}
\eta_{\tau+1}}
\leq\ &
\sqrt{\frac{d}{t}}
\Bigg[
\frac{\alpha_\eta G^2}{3n^2}\ln(2t \!+\! 1)
\!+\!
\frac{8L(f(\bar x(0)) \!\!-\!\! f^\ast)}{\alpha_{\eta}}
\!+\!
\frac{6R_0 L^2 }{(1 \!-\! \rho^2)\sqrt{d}}
\\
&\qquad\qquad
+
\frac{9\alpha_\eta^2\kappa^2\rho^2 G^2}{4(1 \!-\! \rho^2)^2\sqrt{d}}
+\frac{9\alpha_u^2\gamma G^2}{4(\gamma \!-\! 1)}
\Bigg]
+o\!\left(\!\frac{1}{\sqrt{t}}\!\right),
\end{aligned}
\end{equation}
where $\kappa$ is some positive numerical constant, and
\begin{equation}\label{eq:alg1_consensus_2}
\frac{1}{n}\sum_{i=1}^n
\mathbb{E}\!\left[\|x^i(t)-\bar x(t)\|^2\right]
\leq \frac{\alpha_\eta^2\kappa^2\rho^2}{4(1-\rho^2)^2}\frac{G^2/L^2}{t}+o(t^{-1}).
\end{equation}
\end{enumerate}
\end{theorem}

\begin{remark}
Note that in
\eqref{eq:alg1_bound_grad_2}, we use the squared norm of the gradient to assess the sub-optimality of the iterates, and characterize the convergence by \emph{ergodic} rates. This type of convergence rate bound is common for local methods of unconstrained nonconvex problems where we do not aim for global optimal solutions \cite{ghadimi2013stochastic,nesterov2017random}.
\end{remark}

\begin{remark}
Each iteration of Algorithm \ref{alg:2point} requires $2$ queries of function values. Thus the convergence rate
\eqref{eq:alg1_bound_grad_2} can also be interpreted as $O(\sqrt{d/m}\log m)$ where 
$m$ denotes the number of function value queries.
Characterizing convergence rate in terms of the number of function value queries $m$ and the dimension $d$ is conventional for zero-order optimization. In scenarios where zero-order methods are applied, the computation of the function values is usually one of the most time-consuming procedures. In addition, it is also of interest to characterize how the convergence scales with the dimension $d$.
\end{remark}

The next theorem shows that for a gradient dominated global objective, a better convergence rate can be achieved.
{
\begin{theorem}\label{theorem:alg1_grad_dom}
Assume that each local objective function $f_i$ is uniformly $L$-smooth for some $L>0$. Furthermore, assume that $\inf_{x\in\mathbb{R}^d}f_i(x)=f_i^\ast>-\infty$ for each $i$, and that the global objective function $f$ is $\mu$-gradient dominated and has a minimum value denoted by $f^\ast$. Suppose
\begin{equation*}
\eta_t=\frac{2\alpha_\eta}{\mu(t+t_0)},
\quad
u_t\leq\frac{\alpha_u}{\sqrt{t+t_0}}
\end{equation*}
for some $\alpha_\eta>1$ and $\alpha_u>0$, where
$$
t_0\geq
\frac{2\alpha_\eta L}{\mu(1\!-\!\rho^2)}
\!\left(
\frac{32Ld}{3\mu}+
9\rho
\right)
-1.
$$
Then, using Algorithm~\ref{alg:2point}, we have
\begin{align}
\mathbb{E}\!\left[
f(\bar x(t)) \!-\! f^\ast
\right]
\leq\ &
\!\left(\!\frac{32\alpha_\eta^2 L^2\Delta}{\mu^2}
+\frac{6\alpha_\eta\alpha_u^2 L^2}{\mu} \!\right)
\!\frac{d}{t}
+o(t^{-1}), \label{eq:alg1:bound_f_grad_dom} \\
\frac{1}{n} \! \sum_{i=1}^n \! \mathbb{E}\!\left[
\|x^i(t) \!-\! \bar x(t)\|^2\right]
\leq\ &
\frac{32\alpha_\eta^2\rho^2 L\Delta}{\mu^2(1-\rho^2)}
\!\left(
\frac{4d}{3}+\frac{6\rho^2}{1 \!-\!\rho^2}
\right)\!
\frac{1}{t^2} + o(t^{-2}),
\label{eq:alg1_consensus_grad_dom}
\end{align}
where $\Delta\coloneqq
f^\ast-\frac{1}{n}\sum_{i=1}^n f^\ast_i$.
\end{theorem}
}

\begin{remark}
The convergence rate \eqref{eq:alg1:bound_f_grad_dom} can also be described as $\mathbb{E}[f(\bar x(t))-f^\ast]=O(d/m)$, where $m$ is the number of function value queries.
\end{remark}

Table \ref{tab:main_results} shows that, while Algorithm \ref{alg:2point} employs a randomized 2-point zero-order estimator of $\nabla f_i$, its convergence rates are comparable with the decentralized gradient descent (DGD) algorithm \cite{lian2017can,nedic2016stochastic}. However, its convergence rates are inferior to its centralized zero-order counterpart in \cite{nesterov2017random}.

\subsection{Convergence of Algorithm \ref{alg:multipoint}}
Let $(x^i(t),s^i(t))$ denote the sequence generated by Algorithm \ref{alg:multipoint} with a constant step size $\eta$. Denote
$$
\bar x(t)
\coloneqq \frac{1}{n}\sum_{i=1}^n x^i(t),
\qquad
\qquad
R_0
\coloneqq
\frac{1}{n}\sum_{i=1}^n\!
\left(\! \frac{\eta \rho^2 }{2L}\|\nabla f_i(x^i(0))\|^2
\!+\!\|x^i(0)-\bar x(0)\|^2 \!\right)
+\frac{\eta\rho^2 u_1^2 Ld}{4}.
$$
We first analyze the case where the local objectives are nonconvex and smooth.
\begin{theorem}\label{theorem:alg2_non_grad_dom}
Assume that each local objective function $f_i$ is uniformly $L$-smooth for some positive constant $L$, and that $f^\ast\coloneqq \inf_{x\in\mathbb{R}^d} f(x)>-\infty$.
Suppose
$$
\eta L\leq
\min\left\{
\frac{1}{6},
\frac{(1-\rho^2)^2}{4\rho^2(3+4\rho^2)}\right\},
\quad
R_u\coloneqq d\sum_{t=1}^\infty u_t^2
<+\infty,
$$
and that $u_t$ is non-increasing. Then $\lim_{t\rightarrow\infty} f(\bar x(t))$ exists,
\begin{equation}\label{eq:alg2:convergence_alg2}
\begin{aligned}
\frac{1}{t}\sum_{\tau=0}^{t-1}
\|\nabla f(\bar x(\tau))\|^2
\leq\ &
\frac{1}{t}\left[\frac{3.2 (f(\bar x(0))-f^\ast)}{\eta}
+ \frac{12.8 L^2 R_0 }{1-\rho^2}
+2.4 R_u L^2\right],
\end{aligned}
\end{equation}
and
\begin{align}
\frac{1}{t}\sum_{\tau=0}^{t-1}
\frac{1}{n}\sum_{i=1}^n\|x^i(\tau)-\bar x(\tau)\|^2
\leq\ &
\frac{1}{t}\!
\left[
1.6\eta(f(\bar x(0))-f^\ast)
+\frac{3.2 R_0}{1-\rho^2}
+0.35 R_u\right],
\label{eq:alg2:consensus_alg2} \\
\frac{1}{t}\!
\sum_{\tau=1}^{t}\!
\frac{1}{n}\!\sum_{i=1}^n\!\|s^i(\tau)\!-\!\nabla f(\bar x(\tau\!\!-\!\!1))\|^2
\leq\ &
\frac{1}{t}\!
\left[
9.6L(f(\bar x(0))-f^\ast)
+\frac{19.2 L R_0}{\eta(1-\rho^2)}
+\frac{2.35}{\eta}
LR_u
\right]
. \label{eq:alg2:grad_tracking_alg2}
\end{align}
\end{theorem}

\begin{remark}\label{remark:alg2_general}
Theorem \ref{theorem:alg2_non_grad_dom} shows that Algorithm \ref{alg:multipoint} achieves a convergence rate of $O(1/t)$ in terms of the averaged squared norm of $\nabla f(\bar x(t))$, and has a consensus rate of $O(1/t)$ for the averages of the squared consensus error $\|x^i(t) \!-\! \bar x(t)\|^2$ and the squared gradient tracking error $\|s^i(t) \!-\! \nabla f(\bar x(t \!-\! 1))\|^2$. They match the rates for distributed nonconvex optimization with gradient tracking \cite{scutari2019distributed}. On the other hand, since each iteration requires $2d$ queries of function values, we get a $O(d/m)$ rate in terms of the number of function value queries $m$. This matches the convergence rate of centralized zero-order algorithms without Nesterov-type acceleration \cite{nesterov2017random}.
\end{remark}

Now we proceed to the situation with a gradient dominated global objective.
\begin{theorem}\label{theorem:alg2_grad_dom}
Assume that each local objective function $f_i$ is uniformly $L$-smooth for some positive constant $L$, and that the global objective function $f$ is $\mu$-gradient dominated and achieves it global minimum at $x^\ast$. Suppose the step size $\eta$ satisfies
\begin{equation}\label{eq:alg2_grad_dom:eta_condition}
\eta L= \alpha\cdot \left(\frac{\mu}{L}\right)^{\frac{1}{3}}\frac{(1-\rho^2)^2}{14}
\end{equation}
for some $\alpha \in(0,1]$, and $(u_t)_{t\geq 1}$ is non-increasing. Let
$$
\lambda
\coloneqq 1-\alpha\Big(\frac{1-\rho^2}{5}\Big)^2
\Big(\frac{\mu}{L}\Big)^{\frac{4}{3}}.
$$
Then
\begin{equation}\label{eq:lemma_converge_alg2_grad_dom}
f(\bar x(t))-f(x^\ast)
\leq
O(\lambda^t)
+
5(1-\rho^2)Ld
\sum_{\tau=0}^{t-1}
\lambda^\tau u^2_{t-\tau},
\end{equation}
\begin{equation}\label{eq:lemma_consensus_alg2_grad_dom}
\frac{1}{n}\sum_{i=1}^n\| x^i(t) \!-\! \bar x(t)\|^2
\leq
O(\lambda^t)
+\frac{3\eta Ld}{1 \!-\! \rho^2}
\sum_{\tau=0}^{t-1}
\lambda^\tau u^2_{t-\tau},
\end{equation}
\begin{equation}
\frac{1}{n}\sum_{i=1}^n \|s^i(t)\!-\!\nabla f(\bar x(t\!-\!1))\|^2
\leq
O(\lambda^t) + 
\frac{18L^2d}{1\!-\!\rho^2}
\sum_{\tau=0}^{t-1}\lambda^{\tau} u_{t\!-\!\tau}^2.
\end{equation}
\end{theorem}

\begin{remark}
If we use an exponentially decreasing sequence $u_t\propto \tilde{\lambda}^{t/2}$ with $\tilde\lambda<\lambda$, then both the objective error $f(\bar x(t))-f(x^\ast)$ and the consensus errors $\|x^i(t)-\bar x(t)\|^2$ and $\|s^i(t)-\nabla f(\bar x(t-1))\|^2$ achieve linear convergence rate $O(\lambda^t)$, or $O(\lambda^{m/d})$ in terms of the number of function value queries. In addition, we notice that the decaying factor $\lambda$ given by Theorem \ref{theorem:alg2_grad_dom} has a better dependence on $\mu/L$ than in \cite{nedic2017achieving} for convex problems. We point out that this is not a result of using zero-order techniques, but rather a more refined analysis of the gradient tracking procedure.
\end{remark}

{
\begin{remark}
Note that the conditions on the step sizes in Theorems~2, 3 and 4 depend on $\rho$, a measure of the connectivity of the network. In order to choose step sizes to satisfy theses conditions in the distributed setting, one possible approach is as follows: Assuming that each agent knows an upper bound $\overline{n}$ on the total number of agents, by \cite[Lemma 2]{olshevsky2017linear}, if one chooses $W$ to be the lazy Metropolis matrix, then $\rho\leq 1-1/(71\overline{n}^2)$, based on which the agents can then derive their step sizes according to the conditions in the theorems.
We also note that some existing works (e.g., \cite{li2019decentralized}) attempt to get rid of the dependence of step sizes on the graph topology, and whether those techniques can be applied in our work is beyond the scope of this paper but is an interesting future direction.
\end{remark}
}

\subsection{Comparison of the Two Algorithms}\label{subsec:compare}

We see from the above results that Algorithm \ref{alg:multipoint} converges faster than Algorithm \ref{alg:2point} asymptotically as $m\rightarrow\infty$ in theory. However, each iteration of Algorithm \ref{alg:multipoint} makes progress only after $2d$ queries of function values, which could be an issue if $d$ is very large. On the contrary, each iteration of Algorithm \ref{alg:2point} only requires $2$ function value queries, meaning that progress can be made relatively immediately without exploring all the $d$ dimensions. This observation suggests that, when neglecting communication delays, Algorithm \ref{alg:2point} is more favorable for high-dimensional problems, whereas Algorithm \ref{alg:multipoint} could handle problems of relatively low dimensions better with faster convergence.
% This is somewhat analogous to the vanilla gradient descent vs. stochastic gradient descent for machine learning, where we resort to stochastic gradient descent algorithms when the the number of training samples is very large, even though the vanilla gradient descent method theoretically has better convergence rates. 

We emphasize that there still exists a trade-off between the convergence rate and the ability to handle high-dimensional problems  even if one combines the $2$-point gradient estimator \eqref{eq:multipoint_est} with the gradient tracking method as
\begin{equation}\label{eq:alg_2p_grad_track}
\begin{aligned}
g^i(t)
& =
\mathsf{G}^{(2)}_{f_i}
(x^i(t-1);u_t,z^i(t)),\ \ z^i(t)\sim\mathcal{U}(\mathbb{S}_{d-1}), \\
s^i(t) &= \sum_{j=1}^n W_{ij} \!\left( s^j(t\!-\!1) + g^j(t) -g^j(t\!-\!1)\right).
\\
x^i(t) & =
\sum_{j=1}^n W_{ij}(x^j(t-1)
-\eta s^j(t)).
\end{aligned}
\end{equation}
Theoretical analysis suggests that, in order for $s^i(t)$ to reach a consensus in the sense that $\mathbb{E}\!\left[\|s^i(t) \!-\! s^j(t)\|^2\right]\!$ converges to $0$, we need
$$
\lim_{t\rightarrow\infty}\mathbb{E}
\!\left[
\|g^i(t)-g^i(t-1)\|^2\right]\rightarrow 0.
$$
On the other hand, we have the following lemma regarding the variance of the $2$-point gradient estimator $\mathsf{G}^{(2)}_f(x;u,z)$.

\begin{lemma}\label{lemma:2point_var}
Let $f:\mathbb{R}^d\rightarrow\mathbb{R}$ be an arbitrary $L$-smooth function. Then
$$
\lim_{u\rightarrow 0^+}\mathbb{E}_{z}
\!\left[
\|\mathsf{G}^{(2)}_f(x;u,z)\|^2
\right]
=(d-1)\|\nabla f(x)\|^2,
$$
where $z\sim\mathcal{U}(\mathbb{S}_{d-1})$.
\end{lemma}
\begin{proof}
Notice that for any $z\in\mathbb{S}_{d-1}$ and $u\in(0,1]$, we have
$$
\left|\frac{f(x\!+\!uz)-f(x\!-\!uz)}{2u}\right|
\leq
\sup_{y\in\mathbb{B}_d}
\|\nabla f(x+y)\|.
$$
Therefore
$$
\begin{aligned}
\lim_{u\rightarrow 0}
\mathbb{E}_{z}
\!\left[
\|\mathsf{G}^{(2)}_f(x;u,z)\|^2\right]
=\ &
d^2\mathbb{E}_{z}\!
\left[
\left|
\lim_{u\rightarrow 0}
\frac{f(x\!+\!uz) \!-\! f(x\!-\!uz)}{2u}
\right|^2
\right]
\!\!=
d^2\mathbb{E}_{z}\!
\left[
\left|
\nabla f(x)^\top z
\right|^2
\right] \\
=\ &
d^2 \nabla f(x)^\top \mathbb{E}_{z}\!
\left[
zz^\top
\right]\nabla f(x)
= d\|\nabla f(x)\|^2,
\end{aligned}
$$
where in the second step we exchanged the order of limit and expectation by the bounded convergence theorem, and in the last step we used $d\,\mathbb{E}_{z}\!
\left[
zz^\top
\right]=I_d$ for $z \sim\mathcal{U}(\mathbb{S}_{d-1})$. Then, noticing that $\nabla f^u(x)\rightarrow\nabla f(x)$ as $u\rightarrow 0$, we get
$$
\begin{aligned}
\lim_{u\rightarrow 0}
\mathbb{E}_{z}
\!\left[
\|\mathsf{G}^{(2)}_f(x;u,z)
-\nabla f^u(x)\|^2
\right]
=\ &
\lim_{u\rightarrow 0}
\left(
\mathbb{E}_{z}
\!\left[
\|\mathsf{G}^{(2)}_f(x;u,z)\|^2\right]
-\|\nabla f^u(x)\|^2\right) \\
=\ &
(d-1)\|\nabla f(x)\|^2,
\end{aligned}
$$
which completes the proof.
\end{proof}
Lemma~\ref{lemma:2point_var} suggests that, each gradient estimator $\mathsf{G}^{(2)}_{f_i}(x^i(t-1);u_t,z^i(t))$ in \eqref{eq:alg_2p_grad_track} will produce a non-vanishing variance approximately equal to $(d-1)\,\mathbb{E}\!\left[\|\nabla f_i(x^i(t-1))\|^2\right]$ even if we let $u\rightarrow 0$ as $x^i(t)$ approaches a stationary point. Consequently, $\mathbb{E}\!\left[\|g^i(t)-g^i(t-1)\|^2\right]$ is not guaranteed to converge to zero as $t\rightarrow \infty$. The non-vanishing variance will then be reflected in $s^i(t)$ that tracks the global gradient, and consequently the overall convergence will be slowed down. We refer to \cite{nedic2017achieving,qu2018harnessing,pu2018distributed} for related analysis, and to Section~\ref{sec:simulation} for a numerical example.

\subsection{Comparison with Existing Algorithms}\label{subsec:compare_other_algs}

In this subsection, we provide a detailed comparison with existing literature on distributed zero-order optimization, specifically \cite{hajinezhad2019zone,sahu2018distributed,yu2019distributed}.

\begin{enumerate}
\item References \cite{sahu2018distributed,yu2019distributed} discuss convex problems, while \cite{hajinezhad2019zone} and our work focus on nonconvex problems.

\item  In terms of the assumptions on the noisy function queries, \cite{yu2019distributed} and our work consider a noise-free case. \cite{hajinezhad2019zone} considers stochastic queries but assumes two function values can be obtained for a single random sample.
\cite{sahu2018distributed} assumes independent additive noise on each function value query. We expect that our Algorithm~1 can be generalized to the setting adopted in \cite{hajinezhad2019zone} with heavier mathematics. Extensions to general stochastic cases remain our ongoing work.

\item In terms of the approach to reach consensus among agents, our algorithms are similar to \cite{sahu2018distributed,yu2019distributed}, where some weighted average of the neighbors' local variables is utilized, while \cite{hajinezhad2019zone} uses the method of multipliers to design their algorithms. We also mention that, our Algorithm~2 employs the gradient tracking technique, which, to our best knowledge, has not been discussed in existing literature on distributed zero-order optimization yet.

\item Regarding the convergence rates for nonconvex optimization, \cite{hajinezhad2019zone} proved that its proposed ZONE algorithm achieves $O(1/{T})$ rate if each iteration also employs $O(T)$ function value queries, where $T$ is the number of iterations planned in advance. Therefore in terms of the number of function value queries $M$, its convergence rate is in fact $O(1/{\sqrt{M}})$, which is roughly comparable with Algorithm 1 and slower than Algorithm 2 in our paper. Also, \cite{hajinezhad2019zone} did not discuss the dependence on the problem dimension $d$. Moreover, our algorithms only require constant numbers ($2$ or $2d$) of function value queries which is more appealing for practical implementation when $T$ is set to be very large for achieving sufficiently accurate solutions.
\end{enumerate}

\section{Numerical Examples}\label{sec:simulation}

We consider a multi-agent nonconvex optimization problem formulated as
\begin{equation}\label{eq:num_eg}
\begin{aligned}
\min_{x\in\mathbb{R}^d}\ \ 
& \frac{1}{n}\sum_{i=1}^n f_i(x),\\
f_i(x)=\ &
\frac{a_i}{1+\exp(-\xi_i^\top x\!-\!\nu_i)}
+b_i\ln(1+\|x\|^2),
\end{aligned}
\end{equation}
where $a_i,b_i,\nu_i\in\mathbb{R}$ and $\xi_i\in\mathbb{R}^d$ for each $i=1,\ldots,N$.

For the numerical example, we set the dimension to be $d=64$ and the number of agents to be $n=50$. The parameters $a_i$, $\nu_i$ and each entry of $\xi_i$ are randomly generated from the standard Gaussian distribution, and $(b_1,\ldots,b_n)$ is generated from the distribution $\mathcal{N}\big(\mathbf{1}_n,
I_n-\frac{1}{n}\mathbf{1}_n\mathbf{1}_n^\top\big)$ so that $\frac{1}{n}\sum_i b_i=1$. The graph $\mathcal{G}=(\mathcal{N},\mathcal{E})$ is generated by uniformly randomly sampling $n$ points on $\mathbb{S}_2$, and then connecting pairs of points with spherical distances less than $\pi/4$. The Metropolis-Hastings weights \cite{xiao2005scheme} are employed for constructing $W$.

We compare the following algorithms on the problem \eqref{eq:num_eg}:
\begin{enumerate}[itemsep = 0pt]
\item Algorithm~\ref{alg:2point} with $\eta_t=0.02/\sqrt{t}$ and $u_t=4/\sqrt{t}$;
\item Algorithm~\ref{alg:multipoint} with $\eta=0.02$ and $u_t=4/t^{3/4}$;
\item ZONE-M \cite{hajinezhad2019zone}, where we test two setups $J=1$, $\rho_t=4\sqrt{t}$, $u_t=4/\sqrt{t}$ and $J=100$, $\rho_t=0.4\sqrt{t}$, $u_t=4/\sqrt{t}$;
\item $2$-point gradient estimator combined with gradient tracking [see \eqref{eq:alg_2p_grad_track}] with $\eta=2\times 10^{-4}$ and $u_t=4/t^{3/4}$.
\end{enumerate}
All algorithms start from the same initial points, which are randomly generated from the distribution $\mathcal{N}(0,\frac{25}{d}I_{d})$ for each agent.

\subsection{Comparison of Algorithm~1 and Algorithm~2}

\begin{figure}[t]
\centering
\includegraphics[width=.55\textwidth]{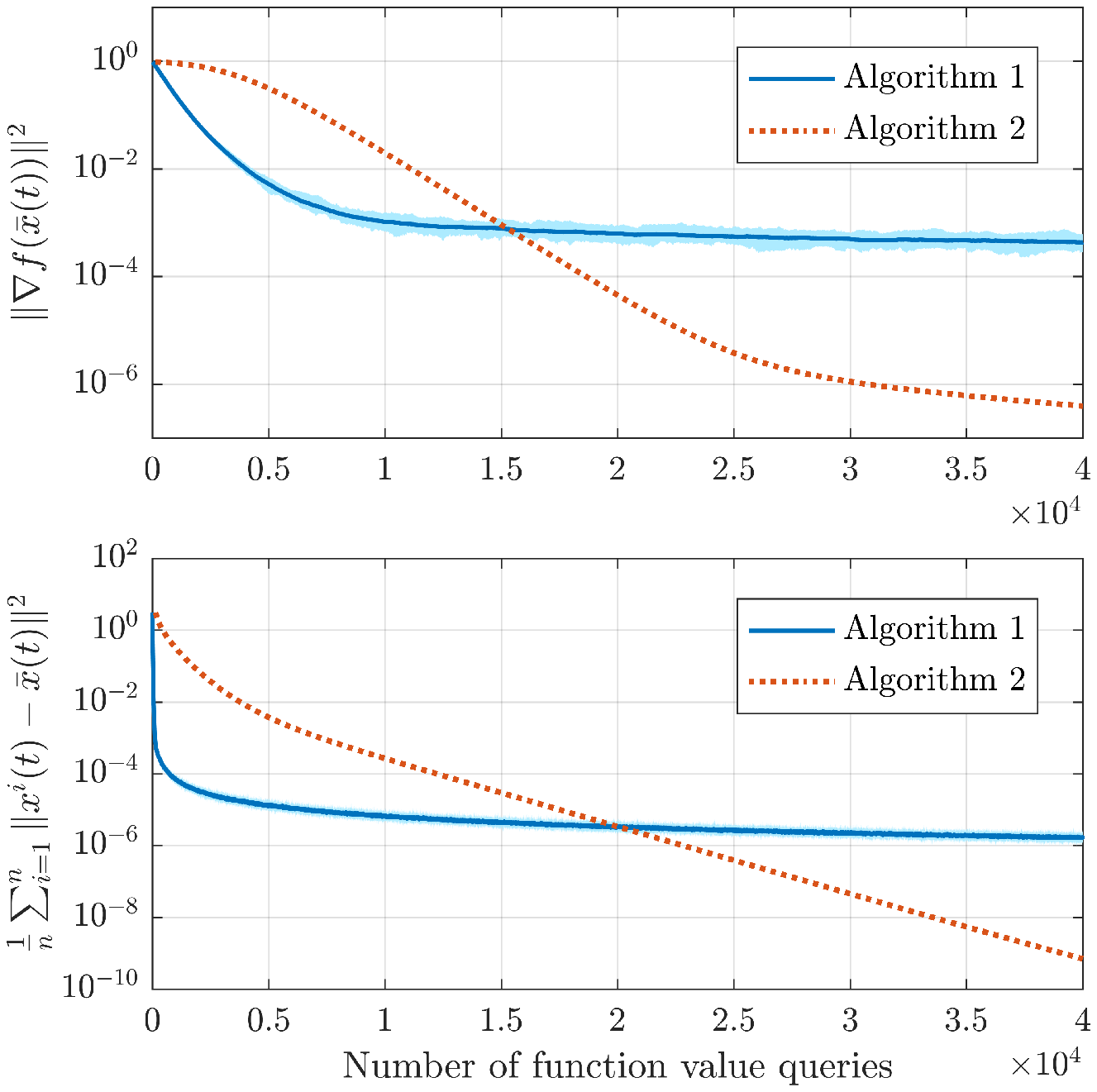}
\caption{Convergence of Algorithm~1 and Algorithm~2. For Algorithm~1, the light blue shaded areas represent the results for $50$ random instances, and the dark blue curves represent their average.}
\label{fig:numerical_alg1_alg2}
\end{figure}

Figure~\ref{fig:numerical_alg1_alg2} shows the convergence behavior of Algorithm~1 and Algorithm~2, where the top figure illustrates the squared norm of the gradient at $\bar x(t)$, and the bottom figure illustrates the consensus error $\frac{1}{n}\sum_{i=1}^n \|x^i(t)-\bar x(t)\|^2$. The horizontal axis has been normalized as the number of function value queries $m$. We can see that Algorithm 1 converges faster during the initial stage, but then slows down and converges at a relatively stable sublinear rate. The convergence of Algorithm 2 is relatively slow initially, but then becomes faster as $m\gtrsim 0.5\times 10^4$, and when $m\gtrsim 2\times 10^4$, Algorithm~2 achieves smaller squared gradient norm and consensus error compared to Algorithm~1; the convergence slows down as $m$ exceeds $2.5\times 10^4$ but is still faster than Algorithm~1. Further investigation of the simulation results suggests that the speed-up of Algorithm~2 within $0.5\times 10^4\lesssim m\lesssim 2.5\times 10^4$ is due to $\bar{x}(t)$ becoming sufficiently close to a local optimal, around which the objective function is locally strongly convex; the slow-down after $m$ exceeds $2.5\times 10^4$ is caused by the zero-order gradient estimation error that becomes dominant, and can be postponed or avoided if we let $u_t$ decrease more aggressively.

From these results, it can be seen that, if the total number of function value queries is limited by, say $m\lesssim 1.5\times 10^4$, then Algorithm 1 might be favorable compared to Algorithm 2 despite slower asymptotic convergence rate, while if more function value queries are allowed, then Algorithm 2 could be favored. We observe that this is related with the discussion in Section \ref{subsec:compare}.

\subsection{Comparison with Other Algorithms}

\begin{figure}[t]
    \centering
\includegraphics[width=.55\textwidth]{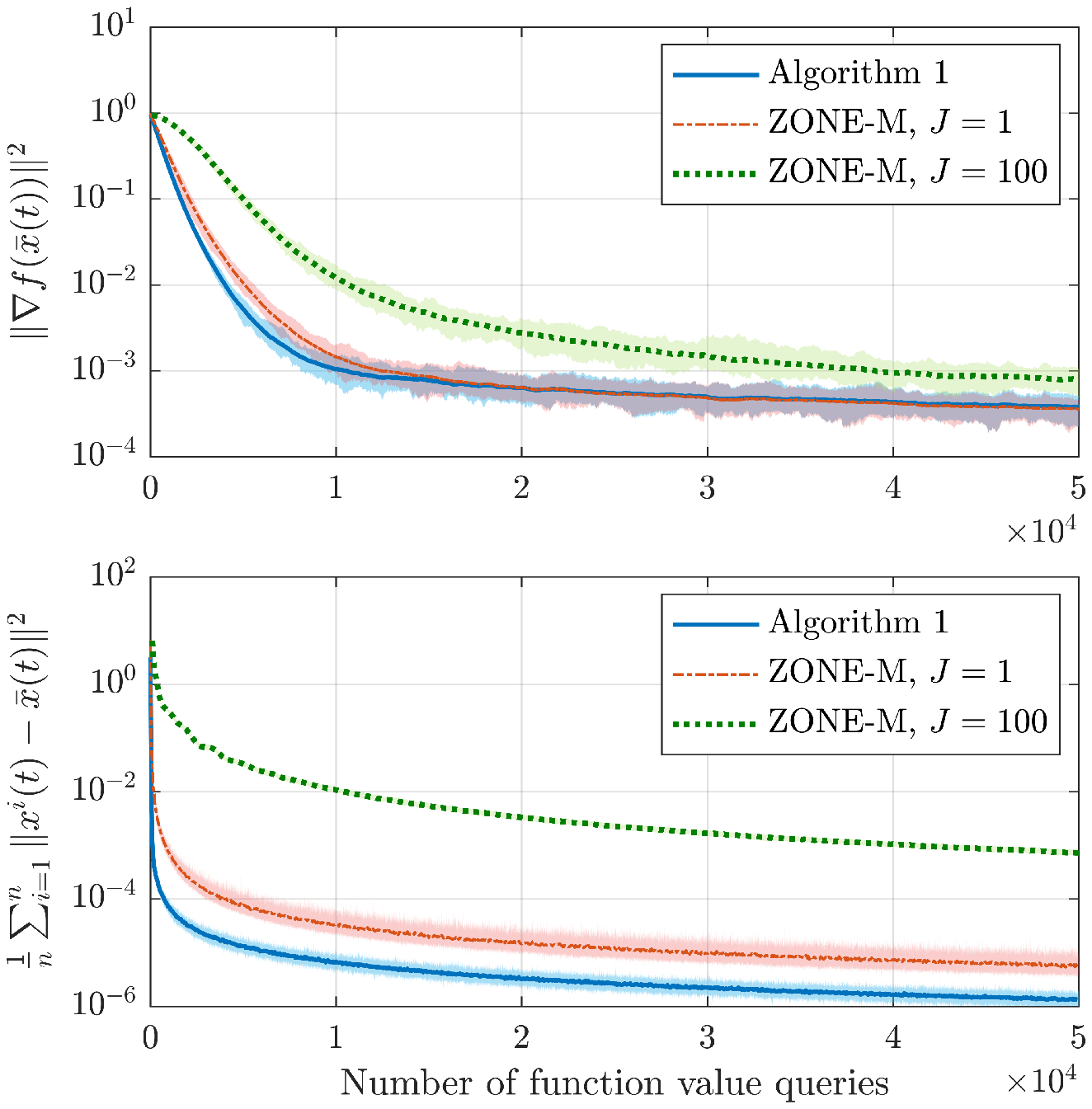}
\caption{Convergence of Algorithm 1 and ZONE-M with $J=1$ and $J=100$. For each algorithm, the light shaded areas represent the results for $50$ random instances, and the dark curves represent their average.}
\label{fig:numerical_alg1_zone}
\end{figure}

Figure~\ref{fig:numerical_alg1_zone} compares the convergence of Algorithm 1 and the two setups of ZONE-M, including the curves for the squared norm of the gradient $\|\nabla f(\bar x(t))\|^2$ and the consensus error $\frac{1}{n}\sum_{i=1}^n \|x^i(t)-\bar x(t)\|^2$. The horizontal axis has been normalized as the number of function value queries $m$. It can be seen that Algorithm 1 and ZONE-M with $\rho_t\propto\sqrt{t}, J=1$ have similar convergence behavior. For ZONE-M with $\rho_t\propto\sqrt{t}$ and $J=100$, while the convergence of $\|\nabla f(\bar x(t))\|^2$ is comparable with Algorithm 1 and ZONE-M with $J=1$, the consensus error decreases much slower, as ZONE-M with $J=100$ conducts much fewer consensus averaging steps per function value query compared to Algorithm~1 and ZONE-M with $J=1$.

\begin{figure}[th]
    \centering
\includegraphics[width=.55\textwidth]{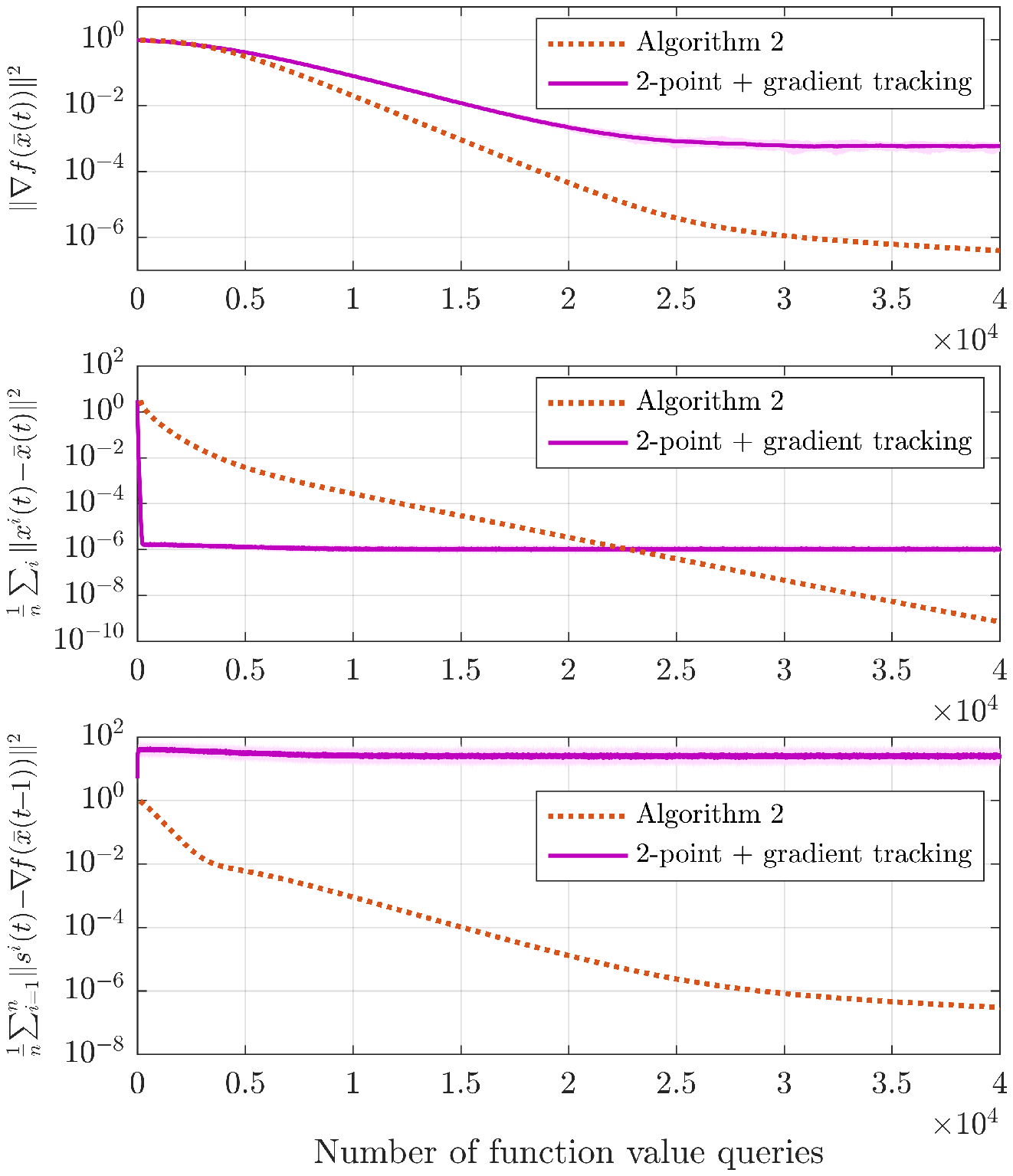}
\caption{Convergence of Algorithm 2 and $2$-point estimator combined with gradient tracking. For $2$-point estimator combined with gradient tracking, the light pink shaded areas represent the results for $50$ random instances, and the dark purple curves represent their average.}
\label{fig:numerical_alg2_comb}
\end{figure}

Figure~\ref{fig:numerical_alg2_comb} compares the convergence of Algorithm~2 and the $2$-point estimator combined with gradient tracking \eqref{eq:alg_2p_grad_track}, including the curves for the squared norm of the gradient $\|\nabla f(\bar x(t))\|^2$, the consensus error $\frac{1}{n}\sum_{i=1}^n \|x^i(t)-\bar x(t)\|^2$ and also the gradient tracking error $\frac{1}{n}\sum_{i=1}^n \|s^i(t)-\nabla f(\bar x(t\!-\!1))\|^2$. It's straightforward to see that Algorithm~2 has better asymptotic convergence behavior than the $2$-point estimator combined with gradient tracking. Moreover, for the $2$-point estimator combined with gradient tracking, the gradient tracking error does not converge to zero but remains at a constant level, indicating that the gradient tracking technique is ineffective in this case. These observations are in accordance with our theoretical discussion in Section~\ref{subsec:compare}.

\section{Conclusion}

We proposed two distribtued zero-order algorithms for nonconvex multi-agent optimization, established theoretical results on their convergence rates, and showed that they achieve comparable performance with their distributed gradient-based or centralized zero-order counterparts. We also provided a brief discussion on how the dimension of the problem will affect their performance in practice. There are many lines of future work, such as 1) introducing noise or errors when evaluating $f_i(x)$, 2) investigating how to escape from saddle-point for distributed zero-order methods, 3) extension to nonsmooth problems, 4) investigating whether the step sizes can be independent of the network topology, 5) studying time-varying graphs, and 6) investigating the fundamental gap between centralized methods and distributed methods, especially for high-dimensional problems.

\appendix

\section{Auxiliary Results for Convergence Analysis}

Recall that $W\in\mathbb{R}^{n\times n}$ is a consensus matrix that satisfies Assumption~\ref{assumption:consensus_weight} in the main text, and
$$
\rho :=
\left\|W-n^{-1}\mathbf{1}_n\mathbf{1}_n^\top\right\|<1.
$$
The following lemma is a standard result in consensus optimization.
\begin{lemma}\label{lemma:consensus_contraction}
For any $x^1,\ldots,x^n\in\mathbb{R}^{d}$, we have
$$
\|(W\otimes I_d)(x-\mathbf{1}_n\otimes\bar x)\|
\leq\rho \|x-\mathbf{1}_n\otimes\bar x\|,
$$
where we denote
$$
x=\begin{bmatrix}
x_1 \\ \vdots \\ x_n
\end{bmatrix},
\qquad
\bar x=\frac{1}{n}\sum_{i=1}^n x^i.
$$
\end{lemma}

The following lemma provides a useful property of smooth functions.
\begin{lemma}\label{lemma:gradient_upper_bound}
Suppose $f:\mathbb{R}^p\rightarrow\mathbb{R}$ is $L$-smooth and $\inf_{x\in\mathbb{R}^p} f(x)=f^\ast>-\infty$. Then
$$
\|\nabla f(x)\|^2\leq 2L(f(x)-f^\ast).
$$
\end{lemma}
\begin{proof}
The $L$-smoothness of $f$ implies
$$
f^\ast\leq f(x-L^{-1}\nabla f(x))
\leq f(x)-\frac{1}{2L}\|\nabla f(x)\|^2.
$$
\end{proof}
For a $\mu$-gradient dominated and $L$-smooth function, we can see from Lemma \ref{lemma:gradient_upper_bound} that $\mu\leq L$.

The following lemma will be used to establish convergence of the proposed algorithms.
\begin{lemma}[\cite{robbins1971convergence}]
\label{lemma:martingale_converge}
Let $(\Omega,\mathcal{F},\mathbb{P})$ be a probability space and $(\mathcal{F}_t)_{t\in\mathbb{N}}$ be a filtration. Let $U(t), \xi(t)$ and $\zeta(t)$ be nonnegative $\mathcal{F}_t$-measurable random variables for $t\in\mathbb{N}$ such that
$$
\mathbb{E}\!\left[U(t+1)|\mathcal{F}_{t}\right]
\leq U(t)+\xi(t)-\zeta(t),\qquad
\forall t=0,1,2,\ldots
$$
Then almost surely on the event $\{\sum_{t=0}^\infty \xi(t)<+\infty\}$, $U(t)$ converges to a random variable and $\sum_{t=0}^\infty \zeta(t)<+\infty$.

As a special case, let $U_t$, $\xi_t$ and $\zeta_t$ be (deterministic) nonnegative sequences for $t\in\mathbb{N}$ such that
$$
U_{t+1}
\leq U_t+\xi_t-\zeta_t,
$$
with $\sum_{t=0}^\infty\xi_t<+\infty$. Then $U_t$ converges and $\sum_{t=0}^\infty\zeta_t<+\infty$.
\end{lemma}

We will extensively use the following properties of the distribution $\mathcal{U}(\mathbb{S}_{d-1})$:
\begin{equation}\label{eq:alg1:basic_uniform_sphere}
\mathbb{E}_{z\sim\mathcal{U}(\mathbb{S}_{d-1})}
\!
\left[d\cdot \langle g,z\rangle z\right]
=g,
\qquad
\mathbb{E}_{z\sim\mathcal{U}(\mathbb{S}_{d-1})}
\!
\left[d\cdot \langle g,z\rangle^2\right]
=\|g\|^2
\end{equation}
for any (deterministic) $g\in\mathbb{R}^d$.

The following results discuss the bias and the second moment of the $2$-point gradient estimator.

\begin{lemma}\label{lemma:alg1:bounds_func}
\begin{enumerate}
\item Let $u> 0$ be arbitrary, and suppose $f:\mathbb{R}^d\rightarrow\mathbb{R}$ is differentiable. Then
\begin{equation}\label{eq:alg1:grad_smoothed_ver}
\mathbb{E}_{z\sim\mathcal{U}(\mathbb{S}_{d-1})}
\!\!
\left[\mathsf{G}^{(2)}_f(x;u,z)\right]
=\nabla f^u(x).
\end{equation}
where $f^u(x):=\mathbb{E}_{y\sim\mathcal{U}(\mathbb{B}_d)}\left[
f(x+uy)
\right]$. Moreover, if $f$ is $L$-smooth, then $f^u$ is also $L$-smooth.

\item Suppose $f:\mathbb{R}^d\rightarrow\mathbb{R}$ is $L$-Lipschitz, and let $u$ be positive. Then for any $x\in\mathbb{R}^d$ and $h\in\mathbb{R}^d$, we have
\begin{equation}\label{eq:alg1:finite_difference}
\left|\frac{f(x+uh)-f(x-uh)}{2u}-\langle \nabla f(x),h\rangle\right|
\leq
\frac{1}{2}uL\|h\|^2.
\end{equation}
In addition,
\begin{equation}\label{eq:alg1:bound_diff_smooth_grad}
\|\nabla f(x)-\nabla f^u(x)\|
\leq uL.
\end{equation}

\item {\cite[Lemma 10]{shamir2017optimal}} Suppose $f:\mathbb{R}^d\rightarrow\mathbb{R}$ is $G$-Lipschitz. Then for any $x\in\mathbb{R}^d$ and $u> 0$,
\begin{equation}\label{eq:alg1:bound_Eg2}
\mathbb{E}_{z\sim\mathcal{U}(\mathbb{S}_{d-1})}
\!\!\left[
\left\|\mathsf{G}^{(2)}_f(x;u,z)\right\|^2\right]
\leq \kappa^2 G^2 d
\end{equation}
where $\kappa>0$ is some numerical constant.

\item Suppose $f:\mathbb{R}^d\rightarrow\mathbb{R}$ is $L$-smooth. Then for any $x\in\mathbb{R}^d$ and $u> 0$,
\begin{equation}\label{eq:alg1:bound_Eg2_ver2}
\mathbb{E}_{z\sim\mathcal{U}(\mathbb{S}_{d-1})}
\!\!\left[
\left\|\mathsf{G}^{(2)}_f(x;u,z)\right\|^2\right]
\leq
\frac{4d}{3}\|\nabla f(x)\|^2
+u^2L^2d^2.
\end{equation}

\end{enumerate}
\end{lemma}
\begin{proof}
\begin{enumerate}
\item The equality \eqref{eq:alg1:grad_smoothed_ver} follows from \cite[Lemma 1]{flaxman2005online} and the fact that the distribution $\mathcal{U}(\mathbb{S}_{d-1})$ has zero mean. When $f$ is $L$-smooth, we have
$$
\begin{aligned}
\left\|\nabla f^u(x_1)- \nabla f^u(x_2)\right\|
& =
\left\|\frac{1}{\int_{\mathbb{B}_d}\,dy}\int_{\mathbb{B}_d}
\left(\nabla f(x_1+uy)
-\nabla f(x_2+uy)
\right)\,dy\right\| \\
&\leq
\frac{1}{\int_{\mathbb{B}_d}\,dy}\int_{\mathbb{B}_d}
\|\nabla f(x_1+uy)-\nabla f(x_2+uy)\|\,dy
\leq L \|x_1-x_2\|
\end{aligned}
$$
for any $x_1,x_2\in\mathbb{R}^d$.

\item[2.]
We have
$$
\begin{aligned}
&\left|\frac{f(x+uh)-f(x-uh)}{2u}-\langle \nabla f(x),h\rangle\right|
=
\left|\frac{1}{2u}\int_{-1}^1\langle\nabla f(x+ush),uh\rangle\,ds
-\langle \nabla f(x),h\rangle\right| \\
=\ &
\frac{1}{2}\left|\int_{-1}^1\langle\nabla f(x+ush)-\nabla f(x),h\rangle\,ds\right|
\leq
\frac{1}{2}
\int_{-1}^1 Lu|s|\|h\|^2\,ds
=\frac{1}{2}uL\|h\|^2,
\end{aligned}
$$
and
$$
\begin{aligned}
\left\|\nabla f(x)-\nabla f^u(x)\right\|
=\ &
\left\|
\frac{1}{\int_{\mathbb{B}_d}\,dy}
\int_{\mathbb{B}_d}
(\nabla f(x)-\nabla f(x+uy))\,dy
\right\|
\leq
\frac{uL}{\int_{\mathbb{B}_d}\,dy}
\int_{\mathbb{B}_d}\|y\|\,dy
\leq uL.
\end{aligned}
$$

\item[4.] We have
$$
\begin{aligned}
& \mathbb{E}_{z\sim\mathcal{U}(\mathbb{S}_{d-1})}\!\left[
\left\|\mathsf{G}_f^{(2)}(x;u,z)\right\|^2\right] \\
=\ &
\mathbb{E}_{z\sim\mathcal{U}(\mathbb{S}_{d-1})}\!\left[
\left\|
d\!\left(\!\frac{f(x\!+\!uz)-f(x\!-\!uz)}{2u}-\langle \nabla f(x),z\rangle\!\right)\! z
+d\langle \nabla f(x),z\rangle z
\right\|^2\right] \\
\leq\ &
(1\!+\!3)\,\mathbb{E}_{z\sim\mathcal{U}(\mathbb{S}_{d-1})}\!\left[d^2
\left|
\frac{f(x \!+\! uz) \!-\! f(x \!-\! uz)}{2u}
-\langle\nabla f(x),z\rangle
\right|^2 \! \|z\|^2\right]\!
+\left(\!1\!+\!\frac{1}{3}\right)
\mathbb{E}_{z\sim\mathcal{U}(\mathbb{S}_{d-1})}\!\left[
\left\|
d\langle\nabla f(x),z\rangle z\right\|^2\right] \\
\leq\ &
4d^2\cdot \frac{1}{4}u^2L^2
+\frac{4d}{3}\|\nabla f(x)\|^2
=\frac{4d}{3}\|\nabla f(x)\|^2
+u^2L^2d^2.
\end{aligned}
$$
\end{enumerate}
\end{proof}

We will also use the following inequalities:
\begin{equation}\label{eq:lower_bound_sum_poly}
\sum_{t=t_1}^{t_2} \frac{1}{t^{\epsilon}}\geq
\int_{t_1}^{t_2+1}\frac{ds}{s^\epsilon}=
\frac{(t_2+1)^{1-\epsilon}-t_1^{1-\epsilon}}{1-\epsilon},
\end{equation}
and
\begin{equation}\label{eq:upper_bound_sum_poly}
\sum_{t=t_1}^{t_2} \frac{1}{t^{\epsilon}}\leq
\left\{
\begin{aligned}
& 1 + \int_{3/2}^{t_2+1/2}\frac{ds}{s^\epsilon}=
1+\frac{(t_2+1/2)^{1-\epsilon}-(3/2)^{1-\epsilon}}{1-\epsilon}, & \quad  & t_1=1,\\
& \int_{t_1-1/2}^{t_2+1/2}\frac{ds}{s^\epsilon}
=\frac{(t_2+1/2)^{1-\epsilon}-(t_1-1/2)^{1-\epsilon}}{1-\epsilon}, & \quad &
t_1>1.
\end{aligned}
\right.
\end{equation}
where $\epsilon>0$ and $\epsilon\neq 1$, and
\begin{equation}\label{eq:bound_harmonic_number}
\ln \frac{t_2+1}{t_1}
= \int_{t_1}^{t_2+1}\frac{ds}{s}
\leq
\sum_{t=t_1}^{t_2}\frac{1}{t}
\leq \int_{t_1-1/2}^{t_2+1/2}\frac{ds}{s}
=\ln\frac{2t_2+1}{2t_1-1}.
\end{equation}
Especially, when $\epsilon>1$, we have
\begin{equation}\label{eq:zeta_upperbound}
\sum_{t=1}^\infty \frac{1}{t^{\epsilon}}\leq
1+
\int_{3/2}^\infty\frac{ds}{s^\epsilon}=
1+\frac{1}{(\epsilon-1)(3/2)^{\epsilon-1}}
\leq\frac{\epsilon}{\epsilon-1}.
\end{equation}
Finally, we note that
\begin{equation}\label{eq:conv_exp_poly_asymp}
\sum_{\tau=0}^{t-1}\frac{\lambda^\tau}{(t-\tau)^\epsilon}=\frac{1}{(1-\lambda)t^\epsilon}+o(t^{-\epsilon})
\end{equation}
for any $\lambda\in(0,1)$ and $\epsilon>0$.

\section{Proof of Theorem~\ref{theorem:alg1_non_grad_dom}}

Let $(\mathcal{F}_t)_{t\in \mathbb{N}}$ be a filtration such that $(z^i(t),x^i(t))$ is $\mathcal{F}_t$-measurable for each $t\geq 1$. We denote
$$
x(t)=\begin{bmatrix}
x^1(t) \\
\vdots \\
x^n(t)
\end{bmatrix},
\quad
g(t)=\begin{bmatrix}
g^1(t) \\
\vdots \\
g^n(t)
\end{bmatrix},
\qquad
\bar x(t)=\frac{1}{n}\sum_{i=1}^n x^i(t),
\quad
\bar g(t)=\frac{1}{n}\sum_{i=1}^n g^i(t),
$$
and $\delta(t)= f(\bar x(t))-f^\ast$, $e_{\mathrm{c}}(t)
= \mathbb{E}\!\left[
\|x(t)-\mathbf{1}_n\otimes\bar{x}(t)\|^2
\right]
$.
We can see that the iterations of Algorithm~\ref{alg:2point} can be equivalently written as
$$
x(t)=(W\otimes I_d)(x(t-1)
-\eta_t g(t)),
\qquad
\bar x(t)=\bar x(t-1)-\eta_t\bar g(t).
$$
We recall that each $f_i$ is assumed to be $G$-Lipschitz and $L$-smooth.

First, we analyze how the objective value at the averaged iterate $f(\bar x(t))$ evolves as the iterations proceed.
\begin{lemma}\label{lemma:alg1:basic_ineq}
We have
\begin{equation}\label{eq:alg1:lemma_basic_ineq}
\begin{aligned}
\mathbb{E}\!\left[
f(\bar x(t))|\mathcal{F}_{t-1}\right]
\leq\ &
f(\bar x(t-1))
-\frac{\eta_t}{2}\|\nabla f(\bar x(t-1))\|^2
+\frac{\eta_t L^2}{n}
\|x(t-1)-\mathbf{1}_n\otimes\bar x(t-1)\|^2 \\
&+\frac{\eta_t^2 L}{2}\mathbb{E}\!\left[\|\bar g(t)\|^2|\mathcal{F}_{t-1}\right]
+\eta_tu_t^2L^2.
\end{aligned}
\end{equation}
\end{lemma}
\begin{proof}
Since $\bar x(t)=\bar x(t-1)-\eta_t\bar g(t)$, by the $L$-smoothness of the function $f$, we get
$$
f(\bar x(t))
\leq f(\bar x(t-1))
-\eta_t\langle\nabla f(\bar x(t-1)),
\bar g(t)\rangle
+\eta_t^2\frac{L}{2}
\|\bar g(t)\|^2.
$$
Note that by \eqref{eq:alg1:grad_smoothed_ver} of Lemma \ref{lemma:alg1:bounds_func}, we have
$$
\mathbb{E}\!\left[\bar g(t)|\mathcal{F}_{t-1}\right]
=\frac{1}{n}\sum_{i=1}^n
\nabla f^{u_t}_i(x^i(t-1)).
$$
By taking the expectation conditioned on $\mathcal{F}_{t-1}$, we get
$$
\begin{aligned}
\mathbb{E}\!\left[f(\bar x(t))|\mathcal{F}_{t-1}\right]
\leq\ &
f(\bar x(t-1))
-\eta_t\|\nabla f(\bar x(t-1))\|^2
+ \eta_t^2\frac{L}{2}\mathbb{E}\!\left[\|\bar g(t)\|^2|\mathcal{F}_{t-1}\right]
\\
&-\eta_t
\left\langle\nabla f(\bar x(t-1)),
\frac{1}{n}\sum_{i=1}^n\left(
\nabla f_i^{u_t}(x^i(t-1))-\nabla f_i^{u_t}(\bar x(t-1))
\right)
\right\rangle
\\
&-\eta_t
\left\langle\nabla f(\bar x(t-1)),
\nabla f^{u_t}(\bar x(t-1))-\nabla f(\bar x(t-1))
\right\rangle.
\end{aligned}
$$
Since each $f_i^{u_t}$ is $L$-smooth (see Part 1 of Lemma \ref{lemma:alg1:bounds_func}), we have
$$
\begin{aligned}
&-\left\langle\nabla f(\bar x(t-1)),
\frac{1}{n}\sum_{i=1}^n\left(
\nabla f_i^{u_t}(x^i(t-1))-\nabla f_i^{u_t}(\bar x(t-1))
\right)
\right\rangle \\
\leq\ &
\frac{1}{2}
\left(
\frac{1}{2}\|\nabla f(\bar x(t-1))\|^2
+
2
\left\|
\frac{1}{n}
\sum_{i=1}^n
\left(\nabla f_i^{u_t}(x^i(t-1))-\nabla f_i^{u_t}(\bar x(t-1))\right)
\right\|^2
\right) \\
\leq\ &
\frac{1}{4}\|\nabla f(\bar x(t-1))\|^2
+
\left(\frac{1}{n}\sum_{i=1}^n L
\|x^i(t-1)-\bar x(t-1)\|\right)^2 \\
\leq\ &
\frac{1}{4}\|\nabla f(\bar x(t-1))\|^2
+\frac{L^2}{n}\|x(t-1)-\mathbf{1}_n\otimes\bar x(t-1)\|^2,
\end{aligned}
$$
and by \eqref{eq:alg1:bound_diff_smooth_grad}, we have
$$
\begin{aligned}
&-\left\langle\nabla f(\bar x(t-1)),
\nabla f^{u_t}(\bar x(t-1))-\nabla f(\bar x(t-1))
\right\rangle \\
\leq\ &
\frac{1}{2}
\left(
\frac{1}{2}\|\nabla f(\bar x(t-1))\|^2
+2\left\|\nabla f^{u_t}(\bar x(t-1))-\nabla f(\bar x(t-1))\right\|^2
\right)
\\
\leq\ &
\frac{1}{4}\|\nabla f(\bar x(t-1))\|^2
+u_t^2 L^2.
\end{aligned}
$$
Therefore
$$
\begin{aligned}
\mathbb{E}\!\left[f(\bar x(t))|\mathcal{F}_{t-1}\right]
\leq\ &
f(\bar x(t-1))
-\frac{\eta_t}{2}\|\nabla f(\bar x(t-1))\|^2
+\eta_t^2\frac{L}{2}\mathbb{E}\!\left[\|\bar g(t)\|^2|\mathcal{F}_{t-1}\right]
\\
&+\frac{\eta_t L^2}{n}
\|x(t-1)-\mathbf{1}_n\otimes\bar x(t-1)\|^2
+\eta_t u_t^2L^2.
\end{aligned}
$$
\end{proof}
Lemma~\ref{lemma:alg1:basic_ineq} suggests that we further need to bound two terms, the second moment of $\bar g(t)$ and the expected consensus error $e_{\mathrm{c}}(t\!-\!1)$.

\begin{lemma}\label{lemma:alg1:gt_square}
We have
$$
\mathbb{E}\!\left[
\|\bar g(t)\|^2
|\mathcal{F}_{t-1}\right]
\leq
\frac{4G^2d}{3n^2}
+2\|\nabla f(\bar x(t-1))\|^2 +\frac{4L^2}{n}\|x(t-1)-\mathbf{1}_n\otimes\bar x(t-1)\|^2
+u_t^2L^2d^2.
$$
\end{lemma}
\begin{proof}
Since
$$
\begin{aligned}
\|\bar g(t)\|^2
=\ &
\Bigg\|\frac{d}{n}\sum_{i=1}^n
\Bigg[
\langle \nabla f_i(x^i(t-1)),z^i(t)\rangle z^i(t) \\
& \qquad
+
\left(\frac{f_i(x^i(t\!-\!1)
\!+\!
u_tz^i(t))
\!-\!
f_i(x^i(t\!-\!1)
\!-\!
u_tz^i(t))}{2u_t}
-\langle\nabla f_i(x^i(t\!-\!1)),z^i(t)\rangle
\!\right)\!z^i(t)
\Bigg]\Bigg\|^2,
\end{aligned}
$$
by \eqref{eq:alg1:finite_difference} of Lemma \ref{lemma:alg1:bounds_func}, we see that
$$
\begin{aligned}
&\mathbb{E}\!\left[\|\bar g(t)\|^2|\mathcal{F}_{t-1}\right] \\
\leq \ &
\mathbb{E}\!\left[
\left.\left(1+\frac{1}{3}\right)
\left(\frac{d}{n}\sum_{i=1}^n
\langle\nabla f_i(x^i(t-1)),z^i(t)\rangle z^i(t)
\right)^2
+(1+3)\left(\frac{d}{n}\sum_{i=1}^n \frac{1}{2}u_t L\right)^2
\right|\mathcal{F}_{t-1}
\right] \\
=\ &
\frac{4}{3}\!
\left(\!
\frac{d}{n^2}\sum_{i=1}^n
\| \nabla f_i(x^i(t-1))\|^2
+
\frac{1}{n^2}\sum_{i\neq j}
\langle \nabla f_i(x^i(t-1)), \nabla f_j(x^j(t-1))\rangle
\!\right)\!
+u_t^2 L^2 d^2,
\end{aligned}
$$
where we used \eqref{eq:alg1:basic_uniform_sphere} and the fact that $\langle\nabla f_i(x^i(t-1)),z^i(t)\rangle z^i(t)$ and $\langle\nabla f_j(x^j(t-1)),z^j(t)\rangle z^j(t)$ are independent for $j\neq i$ conditioned on $\mathcal{F}_{t-1}$. Then since
$$
\begin{aligned}
&\frac{d}{n^2}\sum_{i=1}^n
\|\nabla f_i(x^i(t-1))\|^2
+\frac{1}{n^2}\sum_{i\neq j}
\langle \nabla f_i(x^i(t-1)), \nabla f_j(x^j(t-1))\rangle \\
=\ &
\frac{d-1}{n^2}\sum_{i=1}^n
\|\nabla f_i(x^i(t-1))\|^2
+\left\|\frac{1}{n}\sum_{i=1}^n
\nabla f_i(x^i(t-1))\right\|^2,
\end{aligned}
$$
and
$$
\begin{aligned}
&\left\|\frac{1}{n}\sum_{i=1}^n\nabla f_i(x^i(t-1))\right\|^2 \\
\leq\ &
\left(1+\frac{1}{2}\right)\left\|\frac{1}{n}\sum_{i=1}^n\nabla f_i(\bar x(t-1))\right\|^2
+(1+2)\left\|\frac{1}{n}\sum_{i=1}^n\left(\nabla f_i(x^i(t-1))-\nabla f_i(\bar x(t-1))\right)\right\|^2 \\
\leq\ &
\frac{3}{2}\|\nabla f(\bar x(t-1))\|^2
+3\cdot \frac{1}{n}
\sum_{i=1}^n
\|\nabla f_i(x^i(t-1))-\nabla f_i(\bar x(t-1))\|^2 \\
\leq\ &
\frac{3}{2}\|\nabla f(\bar x(t-1))\|^2
+\frac{3 L^2}{n}
\|x(t-1)-\mathbf{1}_n\otimes\bar x(t-1)\|^2,
\end{aligned}
$$
we get
$$
\begin{aligned}
\mathbb{E}\!\left[\|\bar g(t)\|^2|\mathcal{F}_{t-1}\right]\leq\ &
\frac{4(d-1)}{3n^2}\sum_{i=1}^n \|\nabla f_i(x^i(t-1))\|^2
+\frac{4}{3}\left\|\frac{1}{n}\sum_{i=1}^n\nabla f_i(x^i(t-1))\right\|^2
+u_t^2 L^2 d^2 \\
\leq\ &
\frac{4G^2d}{3n^2}
+
2\|\nabla f(\bar x(t-1))\|^2
+\frac{4L^2}{n}\left\|
x(t-1)-\mathbf{1}\otimes\bar x(t-1)
\right\|^2
+u_t^2 L^2 d^2.
\end{aligned}
$$
\end{proof}

\begin{lemma}\label{lemma:alg1:consensus_error}
For each $t\geq 1$, we have
\begin{equation}\label{eq:alg1:consensus_error_as_bound}
\begin{aligned}
\|x(t)-\mathbf{1}_n\otimes\bar x(t)\|^2
\leq \ &
\left(\frac{1+\rho^2}{2}\right)^t
\|x(0)-\mathbf{1}_n\otimes\bar x(0)\|^2
+
\frac{2n\rho^2}{1-\rho^2}G^2d^2\sum_{\tau=0}^{t-1}
\left(\frac{1+\rho^2}{2}\right)^{\tau}
\eta_{t-\tau}^2
\end{aligned}
\end{equation}
almost surely, and
\begin{equation}\label{eq:alg1:consensus_error_Esq_bound}
\begin{aligned}
e_{\mathrm{c}}(t)
\leq\ &
\left(\frac{1+\rho^2}{2}\right)^t
e_{\mathrm{c}}(0)
+
\frac{2n\rho^2\kappa^2}{1-\rho^2}G^2d
\sum_{\tau=0}^{t-1}
\left(\frac{1+\rho^2}{2}\right)^{\tau}
\eta_{t-\tau}^2.
\end{aligned}
\end{equation}
\end{lemma}
\begin{proof}
We have
$$
x(t)-\mathbf{1}_n\otimes\bar x(t)
=(W\otimes I_d)\left(x(t-1)-
\mathbf{1}_n\otimes\bar x(t-1)
-\eta_t (g(t)
-\mathbf{1}_n\otimes\bar g(t))\right),
$$
and therefore
\begin{equation}\label{eq:alg1:iteration_consensus_error}
\begin{aligned}
\|x(t)-\mathbf{1}_n\otimes\bar x(t)\|^2
\leq\ &
\rho^2
\big(\|x(t-1)-\mathbf{1}_n\otimes\bar x(t-1)\|^2
+\eta_t^2
\|g(t)-\mathbf{1}_n\otimes\bar g(t)\|^2 \\
&
+2\eta_t\|x(t-1)-\mathbf{1}_n\otimes\bar x(t-1)\|
\|g(t)-\mathbf{1}_n\otimes\bar g(t)\|
\big) \\
\leq\ 
&\rho^2
\left(\|x(t-1)-\mathbf{1}_n\otimes\bar x(t-1)\|^2
+\eta_t^2
\|g(t)-\mathbf{1}_n\otimes\bar g(t)\|^2
\right) \\
&
\!\!\!+
\frac{1-\rho^2}{2\rho^2}\cdot \rho^2\|x(t-1)-\mathbf{1}_n\otimes\bar x(t-1)\|^2
+
\frac{2\rho^2}{1-\rho^2}
\cdot \eta_t^2\rho^2\|g(t)-\mathbf{1}_n\otimes\bar g(t)\|^2 \\
=\ &
\frac{1+\rho^2}{2}
\|x(t-1)-\mathbf{1}_n\otimes\bar x(t-1)\|^2
+\eta_t^2
\frac{\rho^2(1+\rho^2)}{1-\rho^2}
\|g(t)-\mathbf{1}_n\otimes\bar g(t)\|^2,
\end{aligned}
\end{equation}
where we used Lemma \ref{lemma:consensus_contraction} in the first inequality. Since each $f_i$ is $G$-Lipschitz, we have $\left\|g^i(t)\right\|\leq Gd$, and therefore
$$
\begin{aligned}
\|g(t)-\mathbf{1}_n\otimes\bar g(t)\|^2
=\ &
\sum_{i=1}^n\left\|g^i(t)-\frac{1}{n}\sum_{j=1}^n g^j(t)\right\|^2
=\sum_{i=1}^n \left\|g^i(t)\right\|^2
-\frac{1}{n}\left\|\sum_{j=1}^n g^j(t)\right\|^2
\\
\leq\ &
\sum_{i=1}^n \|g^i(t)\|^2
\leq nG^2d^2,
\end{aligned}
$$
and by \eqref{eq:alg1:bound_Eg2} of Lemma \ref{lemma:alg1:bounds_func}, we have
$$
\begin{aligned}
\mathbb{E}\!\left[\|g(t)-\mathbf{1}_n\otimes\bar g(t)\|^2|\mathcal{F}_{t-1}\right]
\leq\ &
\mathbb{E}
\left[\left.
\sum_{i=1}^n\left\|g^i(t)\right\|^2
\right|\mathcal{F}_{t-1}\right]
\leq n\kappa^2 G^2 d.
\end{aligned}
$$
By plugging these bounds into \eqref{eq:alg1:iteration_consensus_error} and noting that $\rho<1$, we get \eqref{eq:alg1:consensus_error_as_bound} and \eqref{eq:alg1:consensus_error_Esq_bound} by mathematical induction.
\end{proof}

\begin{corollary}\label{corollary:alg1:converge_consensus}
\begin{enumerate}
\item Let $\eta_t$ be a non-increasing sequence that converges to zero. Then
$$
\lim_{t\rightarrow\infty}
\|x(t)-\mathbf{1}_n\otimes\bar x(t)\|^2
=0.
$$
Furthermore, if $\sum_{\tau=1}^{\infty}\eta_t^3<+\infty$, then
$$
\sum_{t=1}^\infty \eta_t
\|x(t-1)-\mathbf{1}_n\otimes\bar x(t-1)\|^2
<+\infty
$$
almost surely.
\item Suppose $\eta_t=\eta_1/t^\beta$ for $\beta>1/3$. Then
\begin{equation}\label{eq:alg1:sum_sq_consensus_errors}
\sum_{t=1}^\infty \eta_t e_{\mathrm{c}}(t\!-\!1)
\leq
\frac{2\eta_1 e_{\mathrm{c}}(0)}{1-\rho^2}
+\eta_1^3\frac{12\beta n\kappa^2\rho^2}{(3\beta-1)(1-\rho^2)^2} G^2 d.
\end{equation}
\end{enumerate}
\end{corollary}
\begin{proof}
\begin{enumerate}
\item 
By the monotonicity of $\eta_t$ and $((1+\rho^2)/2)^t$, we have
$$
\sum_{\tau=0}^{t-1}
\left(\frac{1+\rho^2}{2}\right)^\tau\eta_{t-\tau}^2
=
\sum_{\tau=1}^{t}
\left(\frac{1+\rho^2}{2}\right)^{t-\tau}\eta_{t}^2
\leq
\sum_{\tau=1}^{t}
\left(\frac{1+\rho^2}{2}\right)^{t-\tau}
\cdot\frac{1}{t}\sum_{\tau=1}^{t}\eta^2_{t}
\longrightarrow 0
$$
as $t\rightarrow\infty$.

For the summability of $\eta_t\|x(t-1)-\mathbf{1}_n\otimes\bar x(t-1)\|^2$, we have
$$
\begin{aligned}
&\sum_{t=2}^\infty
\eta_{t}\|x(t-1)-\mathbf{1}_n\otimes\bar x(t-1)\|^2 \\
\leq\ &
\|x(0)-\mathbf{1}_n\otimes\bar x(0)\|^2\sum_{t=2}^\infty
\eta_t
\left(\frac{1+\rho^2}{2}\right)^{t-1}
+
\frac{2n\rho^2G^2d^2}{1-\rho^2}
\sum_{t=2}^\infty\sum_{\tau=0}^{t-2}\eta_t
\left(\frac{1+\rho^2}{2}\right)^{\tau}\eta_{t-1-\tau}^2 \\
\end{aligned}
$$
The first term on the right-hand side obviously converges. For the second term, we have
$$
\begin{aligned}
\sum_{t=2}^\infty\sum_{\tau=0}^{t-2}\eta_t
\left(\frac{1+\rho^2}{2}\right)^{\tau}\eta_{t-1-\tau}^2
\leq\ &
\sum_{t=2}^\infty\sum_{\tau=0}^{t-2}
\left(\frac{1+\rho^2}{2}\right)^{\tau}\eta_{t-1-\tau}^3
=
\sum_{t=2}^\infty\sum_{\tau=2}^{t}
\left(\frac{1+\rho^2}{2}\right)^{t-\tau}\eta_{\tau-1}^3 \\
=\ &
\sum_{\tau=2}^\infty\eta_{\tau-1}^3\sum_{t=\tau}^\infty
\left(\frac{1+\rho^2}{2}\right)^{t-\tau}
=\frac{2}{1-\rho^2}\sum_{\tau=2}^\infty \eta_{\tau-1}^3<+\infty.
\end{aligned}
$$
Therefore we can conclude that $\eta_t\|x(t-1)-\mathbf{1}_n\otimes\bar x(t-1)\|^2$ is summable almost surely.

\item We have
$$
\begin{aligned}
&\sum_{t=1}^\infty\eta_t
\mathbb{E}[\|x(t-1)-\mathbf{1}_n\otimes\bar x(t-1)\|^2] \\
\leq\ &
\eta_1\|x(0)-\mathbf{1}_n\otimes\bar x(0)\|^2
\sum_{t=1}^\infty
\left(\frac{1+\rho^2}{2}\right)^{t-1}
+
\eta_1^3\frac{2n\rho^2\kappa^2}{1-\rho^2}G^2d\sum_{t=2}^\infty\sum_{\tau=0}^{t-2}\frac{1}{t^\beta(t-1-\tau)^{2\beta}}
\left(\frac{1+\rho^2}{2}\right)^{\tau} \\
\leq\ &
\frac{2\eta_1\|x(0)-\mathbf{1}_n\otimes\bar x(0)\|^2}{1-\rho^2}
+
\eta_1^3\frac{2n\rho^2\kappa^2}{1-\rho^2}G^2d\sum_{t=2}^\infty\sum_{\tau=0}^{t-2}\frac{1}{(t-1-\tau)^{3\beta}}
\left(\frac{1+\rho^2}{2}\right)^{\tau}.
\end{aligned}
$$
Then since
$$
\begin{aligned}
&\sum_{t=2}^\infty\sum_{\tau=0}^{t-2}\frac{1}{(t-1-\tau)^{3\beta}}
\left(\frac{1+\rho^2}{2}\right)^{\tau}
=
\sum_{t=2}^\infty\sum_{\tau=2}^{t}\frac{1}{(\tau-1)^{3\beta}}
\left(\frac{1+\rho^2}{2}\right)^{t-\tau} \\
=\ &
\sum_{\tau=2}^\infty\frac{1}{(\tau-1)^{3\beta}}\sum_{t=\tau}^{\infty}
\left(\frac{1+\rho^2}{2}\right)^{t-\tau}
=\frac{2}{1-\rho^2}
\sum_{\tau=2}^\infty\frac{1}{(\tau-1)^{3\beta}}
\\
\leq\ &
\frac{6\beta}{(3\beta-1)(1-\rho^2)},
\end{aligned}
$$
we get the inequality \eqref{eq:alg1:sum_sq_consensus_errors}.
\end{enumerate}
\end{proof}

Now we are ready to prove Theorem 1 in the main text.

\begin{proof}[Proof of Theorem 1]
Recall that $\delta(t)$ denotes $f(\bar x(t))-f^\ast$. By plugging 
the bound of Lemma \ref{lemma:alg1:gt_square} into \eqref{eq:alg1:lemma_basic_ineq} and noticing that $\eta_t L\leq 1/4$, we get
\begin{equation}\label{eq:alg1:main_bound}
\begin{aligned}
\mathbb{E}\!\left[
\delta(t)|\mathcal{F}_{t-1}\right]
\leq\ &
\delta(t-1)
-\frac{\eta_t}{4}\|\nabla f(\bar x(t-1))\|^2
+\frac{3\eta_t L^2}{2n}
\|x(t-1)-\mathbf{1}_n\otimes\bar x(t-1)\|^2 \\
&+\frac{2\eta_t^2 LG^2d}{3n^2}
+\eta_tu_t^2L^2
\left(1+\frac{1}{2}d^2\eta_t L\right).
\end{aligned}
\end{equation}

\begin{enumerate}
\item Consider the case where $\eta_t$ is non-increasing and $\sum_{t=1}^\infty \eta_t=+\infty$, $\sum_{t=1}^\infty \eta_t^2<+\infty$, and $\sum_{t=1}^\infty \eta_t u_t^2<+\infty$. The convergence of $x^i(t)$ to $\bar x(t)$ is already shown by Corollary \ref{corollary:alg1:converge_consensus}. Moreover, the random variable
$$
\frac{3\eta_t L^2}{2n}
\|x(t-1)-\mathbf{1}_n\otimes\bar x(t-1)\|^2
+\frac{2\eta_t^2 LG^2d}{3n^2}
+\eta_tu_t^2 L^2
\left(1+\frac{1}{2}d^2\eta_t L\right)
$$
is summable almost surely by Corollary \ref{corollary:alg1:converge_consensus} and the assumptions on $\eta_t$ and $u_t$. Then Lemma \ref{lemma:martingale_converge} guarantees that $f(\bar x(t))$ converges and
$$
\sum_{t=1}^\infty \eta_t
\|\nabla f(t-1)\|^2<+\infty 
$$
almost surely, which implies that $\liminf_{t\rightarrow\infty}\|\nabla f(\bar x(t))\|=0$.

Now let $\delta>0$ be arbitrary, and consider the event
$$
A_\delta := \left\{\limsup_{t\rightarrow\infty}
\|\nabla f(\bar x(t))\|\geq\delta\right\}.
$$
On the event $A_\delta$, we can always find a (random) subsequence of $\|\nabla f(\bar x(t))\|$, which we denote by $(\|\nabla f(\bar x(t_k))\|)_{k\in\mathbb{N}}$, such that $\|\nabla f(\bar x(t_k))\|\geq\frac{2\delta}{3}$ for all $k$.
It's not hard to verify that
$$
M:=\sup_{t\geq 1}\|\bar g(t)\|<+\infty.
$$
Then for any $s\geq 1$, we have
$$
\begin{aligned}
\|\nabla f(\bar x(t_k+s))\|
\geq\ &
\|\nabla f(\bar x(t_k))\|
-\sum_{\tau=1}^s\|\nabla f(\bar x(t_k+\tau))-\nabla f(\bar x(t_k+\tau-1))\| \\
\geq\ &
\frac{2\delta}{3}
-\sum_{\tau=1}^s L\cdot\eta_{t_k+\tau}M
\end{aligned}
$$
Let $\hat s(k)$ be the smallest positive integer such that
$$
\frac{2\delta}{3}
-\sum_{\tau=1}^{\hat s(k)+1} L\cdot\eta_{t_k+\tau}M
<\frac{\delta}{3}
$$
(such $\hat s(k)$ exists as $\sum_{t=1}^\infty \eta_t=+\infty$). We then see that
$$
\sum_{\tau=1}^{\hat s(k)+1} \eta_{t_k+\tau}
>\frac{\delta}{3LM}
\qquad
\textrm{and}
\qquad 
\|\nabla f(\bar x(t_k+s))\|\geq\frac{\delta}{3}
$$
for all $s=0,\ldots,\hat s(k)$. Therefore
$$
\sum_{\tau=1}^{\hat s(k)+1}\eta_{t_k+\tau}
\|\nabla f(\bar x(t_k+\tau-1)))\|^2
\geq \sum_{\tau=1}^{\hat s(k)+1}
\eta_{t_k+\tau}
\frac{\delta^2}{9}
\geq \frac{\delta^3}{27LM}
$$
Since $t_k\rightarrow\infty$ as $k\rightarrow\infty$, we can find a subsequence of $(t_{k_p})_{p\in\mathbb{N}}$ satisfying
$t_{k_{p+1}}-t_{k_p}>\hat s(k_p)$ by induction, and then
$$
\sum_{t=1}^\infty \eta_t\|\nabla f(\bar x(t-1))\|^2
\geq
\sum_{p=0}^\infty \frac{\delta^3}{27LM}=+\infty.
$$
In other words, on $A_\delta$ the series $\sum_{t=1}^\infty\eta_t\|\nabla f(t-1)\|^2$ diverges. Since $\sum_{t=1}^\infty\eta_t\|\nabla f(t-1)\|^2<+\infty$ converges almost surely, we have $\mathbb{P}(A_\delta)=0$, and consequently
$$
\mathbb{P}\left(
\limsup_{t\rightarrow\infty}
\|\nabla f(\bar x(t))\|>0
\right)
=\mathbb{P}
\left(
\bigcup_{k=1}^\infty
A_{1/k}
\right)
=\lim_{k\rightarrow\infty} \mathbb{P}(A_{1/k})=0,
$$
and we see that $\|\nabla f(\bar x(t))\|$ converges almost surely.

\item When $\eta_t=\eta_1/\sqrt{t}$ and $u_t=u_1/t^{\gamma/2-1/4}$, by \eqref{eq:alg1:main_bound} we have
$$
\begin{aligned}
\mathbb{E}\!\left[
\left.\frac{\delta(t)}{(t+1)^\epsilon} \right|\mathcal{F}_{t-1}\right]
\leq\ &
\frac{1}{t^\epsilon}\delta(t-1)
-\frac{\eta_1}{4t^{1/2+\epsilon}}\|\nabla f(\bar x(t-1))\|^2 \\
&+\frac{3\eta_t L^2}{2nt^{\epsilon}}
\|x(t-1)-\mathbf{1}_n\otimes\bar x(t-1)\|^2
+\frac{2\eta_1^2 LG^2d}{3n^2 t^{1+\epsilon}}
+\frac{\eta_tu_t^2L^2 }{t^\epsilon}\left(1+\frac{1}{2}d^2\eta_t L\right),
\end{aligned}
$$
where $\epsilon>0$ is arbitrary. Since
$$
\frac{3\eta_t L^2}{2nt^{\epsilon}}
\|x(t-1)-\mathbf{1}_n\otimes\bar x(t-1)\|^2 +\frac{2\eta_1^2 LG^2(d-1)}{3n^2 t^{1+\epsilon}}
+\frac{\eta_tu_t^2L^2 }{t^\epsilon}\left(1+\frac{1}{2}d^2\eta_t L\right)
$$
is summable, we see that
$$
\sum_{t=1}^\infty\frac{\eta_1}{t^{1/2+\epsilon}}
\|\nabla f(\bar x(t-1))\|^2<+\infty,
$$
which implies that
$$
\liminf_{t\rightarrow\infty}\|\nabla f(\bar x(t))\|=0.
$$

Now by taking the telescoping sum of \eqref{eq:alg1:main_bound} and noting that $\delta(t)\geq 0$, we get
\begin{equation}\label{eq:alg1_proof_temp1}
\begin{aligned}
\sum_{\tau=1}^t \eta_\tau \mathbb{E}\!\left[\|\nabla f(\bar x(t-1))\|^2\right]
\leq\ &
4\delta(0)
+\frac{6L^2}{n}
\sum_{\tau=1}^t \eta_\tau
e_{\mathrm{c}}(\tau\!-\!1)
+\frac{8LG^2d}{3n^2}
\sum_{\tau=1}^t\eta_\tau^2 \\
&+4L^2
\sum_{\tau=1}^t
\! \left(\eta_\tau u_\tau^2
+\frac{1}{2}d^2L \eta_\tau^2 u_\tau^2
\right).
\end{aligned}
\end{equation}
Since $\eta_t=\eta_1/\sqrt{t}=\alpha_\eta/(4L\sqrt{d\cdot t})
\leq 1/(4L\sqrt{d})$ and $u_t\leq \alpha_u G/(L\sqrt{d}t^{\gamma/2-1/4})$ with $\alpha_\eta\leq 1$ and $\gamma>1$, we have
$$
\begin{aligned}
\sum_{\tau=1}^t \eta_\tau
=\ &
\eta_1 \sum_{\tau=1}^t\frac{1}{\sqrt{t}}
\geq 2\eta_1(\sqrt{t+1}-1), \\
\sum_{\tau=1}^t \eta_\tau^2
=\ &
\eta_1^2 \sum_{\tau=1}^t\frac{1}{t}
\leq \eta_1^2 \ln(2t+1), \\
\sum_{\tau=1}^t
\left(\eta_\tau u_\tau^2+\frac{1}{2}d^2 L\eta_\tau^2 u_\tau^2\right)
\leq\ &
\sum_{\tau=1}^t
\left(1+\frac{d^{3/2}}{8}\right)\eta_\tau u_\tau^2 \\
\leq\ &
\frac{9d^{3/2}}{8}\cdot\eta_1
\frac{\alpha_u^2 G^2}{L^2 d}\sum_{\tau=1}^t
\frac{1}{t^{\gamma}}
\leq
\eta_1\frac{9\alpha_u^2 G^2\sqrt{d}}{8L^2}\frac{\gamma}{\gamma-1},
\end{aligned}
$$
where we used \eqref{eq:lower_bound_sum_poly}, \eqref{eq:bound_harmonic_number} and \eqref{eq:zeta_upperbound}.
In addition, by the second part of Corollary \ref{corollary:alg1:converge_consensus}, we get
\begin{align*}
\sum_{\tau=1}^t \eta_\tau
\frac{e_{\mathrm{c}}(\tau-1)}{n}
\leq\ &
\frac{2\eta_1 e_{\mathrm{c}}(0)}{n(1\!-\!\rho^2)}
+\eta_1^3\frac{12 \kappa^2\rho^2}{(1-\rho^2)^2} G^2 d.
\end{align*}
By plugging these bounds into \eqref{eq:alg1_proof_temp1}, we get
$$
\begin{aligned}
\frac{\sum_{\tau=1}^t \eta_\tau \mathbb{E}\!\left[\|\nabla f(\bar x(t-1))\|^2\right]}{\sum_{\tau=1}^t\eta_\tau}
\leq\ &
\frac{1}{\sqrt{t+1}-1}
\left(
\frac{8\sqrt{d} L \delta(0)}{\alpha_\eta}
+\frac{6 L^2 e_{\mathrm{c}}(0)}{n(1-\rho^2)}
+\frac{9\alpha_u^2\gamma}{4(\gamma-1)}G^2\sqrt{d}
\right) \\
&
+\frac{\alpha_\eta G^2\sqrt{d}}{3n^2}\frac{\ln(2t+1)}{\sqrt{t+1}-1}
+\frac{9\kappa^2\rho^2}{4(1-\rho^2)^2}\frac{\alpha_\eta^2 G^2}{\sqrt{t+1}-1}.
\end{aligned}
$$
We now get the convergence rate of $\mathbb{E}\!\left[\|\nabla f(\bar x(t))\|^2\right]$ stated in the theorem. The convergence rate of the consensus error follows from Lemma \ref{lemma:alg1:consensus_error} and \eqref{eq:conv_exp_poly_asymp}.
\end{enumerate}
\end{proof}

\section{Proof of Theorem~\ref{theorem:alg1_grad_dom}}

We still denote
$$
x(t)=\begin{bmatrix}
x^1(t) \\
\vdots \\
x^n(t)
\end{bmatrix},
\quad
g(t)=\begin{bmatrix}
g^1(t) \\
\vdots \\
g^n(t)
\end{bmatrix},
\qquad
\bar x(t)=\frac{1}{n}\sum_{i=1}^n x^i(t),
\quad
\bar g(t)=\frac{1}{n}\sum_{i=1}^n g^i(t),
$$
and $\delta(t)= f(\bar x(t))-f^\ast$, $e_{\mathrm{c}}(t)
= \mathbb{E}\!\left[
\|x(t)-\mathbf{1}_n\otimes\bar{x}(t)\|^2
\right]
$,
and $(\mathcal{F}_t)_{t\in \mathbb{N}}$ will be a filtration such that $(z^i(t),x^i(t))$ is $\mathcal{F}_t$-measurable for each $t\geq 1$.
We recall that the iterations of Algorithm~\ref{alg:2point} can be equivalently written as
$$
x(t)=(W\otimes I_d)(x(t-1)
-\eta_t g(t)),
\qquad
\bar x(t)=\bar x(t-1)-\eta_t\bar g(t),
$$
and that each $f_i$ is assumed to be $L$-smooth and $f_i^\ast=\inf_{x\in\mathbb{R}^d} f_i(x)>-\infty$. In addition, $\Delta$ is defined as $\Delta\coloneqq f^\ast-\frac{1}{n}\sum_{i=1}^n f_i^\ast$.

Note  that  Lemma~\ref{lemma:alg1:basic_ineq}  still  applies  here.  On  the  other  hand,as each $f_i$ is not uniformly Lipschitz continuous over $\mathbb{R}^d$, we need new lemmas characterizing the consensus procedure and the second moment of $\bar{g}(t)$.

\begin{lemma}\label{lemma:alg1_graddom:consns_err_bound}
Suppose
\begin{equation}\label{eq:alg1_graddom:eta_cond1}
\eta_t^2 L^2
\leq \frac{1-\rho^2}{12\rho^2\left(
\frac{4d}{3}+\frac{6\rho^2}{1-\rho^2}\right)}.
\end{equation}
Then for each $t\geq 1$,
\begin{equation}\label{eq:alg1_graddom:consns_err_one_step}
\begin{aligned}
\frac{e_{\mathrm{c}}(t)}{n}
\leq\ &
\frac{1 + \rho^2}{2}
\frac{e_{\mathrm{c}}(t-1)}{n}
+
4\eta_t^2
\rho^2L
\!\left(
\frac{4d}{3}
+\frac{6\rho^2}{1 - \rho^2}
\right)\!
\mathbb{E}\!\left[\delta(t-1)\right] \\
&
+
4\eta_t^2
\rho^2L\Delta
\!\left(
\frac{4d}{3}
+\frac{6\rho^2}{1-\rho^2}
\!\right)
+
\eta_t^2\rho^2 u_t^2 L^2
\!\left(
d^2
+
\frac{6\rho^2}{1 - \rho^2}
\right)\!.
\end{aligned}
\end{equation}
Consequently,
\begin{equation}\label{eq:alg1_graddom:consns_err_bound}
\begin{aligned}
\frac{e_{\mathrm{c}}(t)}{n}
\leq\ &
\frac{16\alpha_\eta^2\rho^2 L}{\mu^2}
\left(\frac{4d}{3}
+\frac{6\rho^2}{1\!-\!\rho^2}\right)
\sum_{\tau=1}^t
\frac{\mathbb{E}[\delta(\tau\!-\!1)]}
{(\tau+t_0)^2}
\left(\frac{1+\rho^2}{2}\right)^{t-\tau} \\
& +
\frac{32\alpha_\eta^2
\rho^2 L\Delta}{\mu^2(1-\rho^2)}\left(\frac{4d}{3}
+\frac{6\rho^2}{1\!-\!\rho^2}\right)
\frac{1}{t^2}
+ o(t^{-2}).
\end{aligned}
\end{equation}
\end{lemma}
\begin{proof}
From
$
x(t)-\mathbf{1}_n\otimes\bar{x}(t)
=(W\otimes I_d)(x(t \!-\! 1)-\mathbf{1}_n\otimes\bar{x}(t \!-\! 1)
-\eta_t(g(t)-\mathbf{1}_n\otimes\bar g(t)))
$,
we get
$$
\begin{aligned}
\left\|x(t)-\mathbf{1}_n\otimes\bar{x}(t)\right\|^2
\leq\ &
\rho^2
\left\|x(t \!-\! 1)-\mathbf{1}_n\otimes\bar{x}(t\!-\! 1)
-\eta_t(g(t)-\mathbf{1}_n\otimes\bar g(t))\right\|^2 \\
\leq\ &
\rho^2\left\|x(t\!-\!1)-\mathbf{1}_n\!\otimes\!\bar{x}(t\!-\!1)\right\|^2
+\rho^2\eta_t^2 \left\|g(t)-\mathbf{1}_n\otimes\bar g(t)\right\|^2 \\
&
-2\rho^2\eta_t\left\langle x(t\!-\!1)-\mathbf{1}_n\otimes\bar{x}(t\!-\!1),
g(t)-\mathbf{1}_n\otimes\bar g(t)\right\rangle.
\end{aligned}
$$
By $\|g(t)-\mathbf{1}_n\otimes\bar{g}(t)\|\leq \|g(t)\|$ and the bound \eqref{eq:alg1:bound_Eg2_ver2} of Lemma~\ref{lemma:alg1:bounds_func}, we get
$$
\begin{aligned}
\mathbb{E}\!\left[\left\|g(t)-\mathbf{1}_n\otimes \bar g(t)\right\|^2|\mathcal{F}_{t-1}\right]
\leq \ &
\sum_{i=1}^n\mathbb{E}\!\left[
\left.\left\|g^i(t)\right\|^2\right|\mathcal{F}_{t-1}\right]
\leq
\frac{4d}{3}\sum_{i=1}^n
\|\nabla f_i(x^i(t\!-\!1))\|^2
+nu_t^2L^2d^2.
\end{aligned}
$$
On the other hand, we have
$$
\begin{aligned}
& -
2\eta_t\,\mathbb{E}\!\left[\left\langle x(t\!-\!1)-\mathbf{1}_n\otimes\bar{x}(t\!-\!1),
g(t)-\mathbf{1}_n\otimes\bar g(t)\right\rangle
|\mathcal{F}_{t-1}\right] \\
\leq\ &
\frac{1\!-\!\rho^2}{3\rho^2}
\left\|x(t\!-\!1)-\mathbf{1}_n\otimes\bar{x}(t\!-\!1)\right\|^2
+
\frac{3\rho^2}{1\!-\!\rho^2}\eta_t^2
\sum_{i=1}^n
\left\|
\nabla f_i^{u_t}(x^i(t\!-\!1))
-
\frac{1}{n}\sum_{j=1}^n
\nabla f_j^{u_t}(x^j(t\!-\!1))
\right\|^2 \\
\leq\ &
\frac{1\!-\!\rho^2}{3\rho^2}
\left\|x(t\!-\!1)-\mathbf{1}_n\otimes\bar{x}(t\!-\!1)\right\|^2
+
\frac{3\rho^2}{1\!-\!\rho^2}\eta_t^2
\sum_{i=1}^n\|\nabla f_i^{u_t}(x^i(t\!-\!1))\|^2 \\
\leq\ &
\frac{1 \!-\! \rho^2}{3\rho^2}
\left\|x(t \!-\! 1)-\mathbf{1}_n\otimes\bar{x}(t \!-\! 1)\right\|^2
+
\frac{6\rho^2}{1\!-\!\rho^2}\eta_t^2
\sum_{i=1}^n\|\nabla f_i(x^i(t\!-\!1))\|^2
+ \frac{6\rho^2}{1\!-\!\rho^2}\eta_t^2 n u_t^2L^2.
\end{aligned}
$$
Then, we notice that
$$
\begin{aligned}
\sum_{i=1}^n \|\nabla f_i(x^i(t \!-\! 1))\|^2
\leq\ &
2\sum_{i=1}^n
\left(\|\nabla f_i(x^i(t \!-\! 1))
-\nabla f_i(\bar{x}(t \!-\! 1))\|^2
+
\|\nabla f_i(\bar{x}(t \!-\! 1))\|^2\right) \\
\leq\ &
2L^2
\|x(t \!-\! 1)
-\mathbf{1}_n\otimes\bar{x}(t \!-\! 1)\|^2
+
4L\sum_{i=1}^n
(f_i(\bar x(t\!-\!1))-f_i^\ast) \\
\leq\ &
2L^2
\|x(t \!-\! 1)
-\mathbf{1}_n\otimes\bar{x}(t \!-\! 1)\|^2
+
4Ln(f(\bar x(t\!-\!1))
-f^\ast)
+4Ln\Delta,
\end{aligned}
$$
where we used the $L$-smoothness of each $f_i$ and Lemma~\ref{lemma:gradient_upper_bound}. Therefore
$$
\begin{aligned}
& \mathbb{E}\!\left[
\left\|x(t)-\mathbf{1}_n\otimes\bar{x}(t)\right\|^2
|\mathcal{F}_{t-1}\right] \\
\leq\ &
\frac{1+2\rho^2}{3}
\left\|x(t\!-\!1)-\mathbf{1}_n\otimes\bar{x}(t\!-\!1)\right\|^2
+
\eta_t^2\rho^2
\left(\frac{4d}{3}\sum_{i=1}^n \|\nabla f_i(x^i(t \!-\! 1))\|^2
+nu_t^2L^2d^2\right) \\
&
+\frac{6\rho^4}{1-\rho^2}\eta_t^2
\sum_{i=1}^n\|\nabla f_i(x^i(t \!-\! 1))\|^2
+\frac{6\rho^4}{1-\rho^2}\eta_t^2 nu_t^2L^2 \\
\leq\ &
\frac{1+2\rho^2}{3}
\left\|x(t \!-\! 1)-\mathbf{1}_n\otimes\bar{x}(t \!-\! 1)\right\|^2
+
\eta_t^2
\rho^2
\left(\frac{4d}{3}
+\frac{6\rho^2}{1-\rho^2}\right)
\sum_{i=1}^n \|\nabla f_i(x^i(t \!-\! 1))\|^2 \\
&
+ \eta_t^2\rho^2 nu_t^2 L^2
\left(d^2
+\frac{6\rho^2}{1-\rho^2}\right) \\
\leq\ &
\left[
\frac{1+2\rho^2}{3}
+2\eta_t^2 L^2\rho^2
\left(\frac{4d}{3}
+\frac{6\rho^2}{1-\rho^2}\right)
\right]
\left\|x(t \!-\! 1)-\mathbf{1}_n\otimes\bar{x}(t \!-\! 1)\right\|^2 \\
&
+
\eta_t^2
\rho^2
\left(\frac{4d}{3}
+\frac{6\rho^2}{1-\rho^2}\right)
\cdot 4Ln(f(\bar x(t\!-\!1))-f^\ast)
+
\eta_t^2
\rho^2
\left(\frac{4d}{3}
+\frac{6\rho^2}{1-\rho^2}\right)
\cdot 4Ln\Delta \\
& +
\eta_t^2\rho^2 nu_t^2 L^2
\left(d^2
+\frac{6\rho^2}{1-\rho^2}\right).
\end{aligned}
$$
Since the condition \eqref{eq:alg1_graddom:eta_cond1} implies
$$
\frac{1+2\rho^2}{3}
+2\eta_t^2 L^2\rho^2
\left(\frac{4d}{3}
+\frac{6\rho^2}{1-\rho^2}\right)
\leq\frac{1+\rho^2}{2},
$$
by taking the total expectation of the bound on $\mathbb{E}\!\left[
\left\|x(t)-\mathbf{1}_n\otimes\bar{x}(t)\right\|^2
|\mathcal{F}_{t-1}\right]$, we get \eqref{eq:alg1_graddom:consns_err_one_step}.
Finally, by induction, we get
$$
\begin{aligned}
\frac{e_{\mathrm{c}}(t)}{n}
\leq\ &
\left(\frac{1+\rho^2}{2}\right)^t
\frac{e_{\mathrm{c}}(0)}{n}
+
4\rho^2 L
\left(\frac{4d}{3}
+\frac{6\rho^2}{1\!-\!\rho^2}\right)
\sum_{\tau=1}^t\eta_\tau^2\,
\mathbb{E}[\delta(\tau\!-\!1)]
\left(\frac{1+\rho^2}{2}\right)^{t-\tau} \\
&+
4\rho^2 L\left(\frac{4d}{3}
+\frac{6\rho^2}{1\!-\!\rho^2}\right)\Delta
\sum_{\tau=1}^t\eta_\tau^2
\!\left(\!\frac{1\!+\!\rho^2}{2}\!\right)^{t-\tau}
+
\rho^2L^2\left(d^2
+\frac{6\rho^2}{1\!-\!\rho^2}\right)
\sum_{\tau=1}^t\eta_\tau^2u_\tau^2
\left(\!\frac{1\!+\!\rho^2}{2}\!\right)^{t-\tau} \\
=\ &
\frac{16\alpha_\eta^2\rho^2 L}{\mu^2}
\!\left(\frac{4d}{3}
\!+\!\frac{6\rho^2}{1\!-\!\rho^2}\right)
\!\sum_{\tau=1}^t
\frac{\mathbb{E}[\delta(\tau\!-\!1)]}
{(\tau+t_0)^2}
\!\left(\frac{1 \!+\! \rho^2}{2}\right)^{\! t-\tau}
\!\!+
\frac{32\alpha_\eta^2
\rho^2 L\Delta}{\mu^2(1 \!-\! \rho^2)}
\!\left(\frac{4d}{3}
\!+\! \frac{6\rho^2}{1\!-\!\rho^2}\right)\!
\frac{1}{t^2}
+ o(t^{-2}),
\end{aligned}
$$
which completes the proof.
\end{proof}

\begin{lemma}
We have
\begin{equation}\label{eq:alg1_graddom:gt_second_moment}
\mathbb{E}\!\left[
\|\bar g(t)\|^2
\right]
\leq
\frac{8L^2d}{n} e_{\mathrm{c}}(t-1)
+\frac{32Ld}{3}
\mathbb{E}\!\left[\delta(t-1)\right]
+\frac{32Ld\Delta}{3}
+u_t^2L^2d^2.
\end{equation}
\end{lemma}
\begin{proof}
We have
$$
\begin{aligned}
\mathbb{E}\!\left[\left.\|\bar g(t)\|^2\right|\mathcal{F}_{t-1}\right]
=\ &
\mathbb{E}\!\left[
\left.
\left\|
\frac{1}{n}\sum_{i=1}^n
\mathsf{G}^{(2)}_{f_i}(x^i(t\!-\!1);u_t,z^i(t))\right\|^2
\right|\mathcal{F}_{t-1}
\right] \\
\leq\ &
\frac{1}{n}
\sum_{i=1}^n
\mathbb{E}\!\left[
\left.
\left\|
\mathsf{G}^{(2)}_{f_i}(x^i(t\!-\!1);u_t,z^i(t))\right\|^2
\right|\mathcal{F}_{t-1}
\right] \\
\leq\ &
\frac{4d}{3n}
\sum_{i=1}^n
\|\nabla f_i(x^i(t\!-\!1))\|^2
+u_t^2L^2d^2 \\
\leq\ &
\frac{8d}{3n}
\sum_{i=1}^n
\|\nabla f_i(\bar x(t\!-\!1))\|^2
+\frac{8L^2d}{3n}
\|x(t\!-\!1)-\mathbf{1}_n\otimes\bar{x}(t\!-\!1)\|^2
+u_t^2L^2d^2 \\
\leq\ &
\frac{8L^2d}{n}\|x(t\!-\!1)-\mathbf{1}_n\otimes\bar{x}(t\!-\!1)\|^2
+\frac{32Ld}{3}
(f(\bar x(t\!-\!1))-f^\ast)
+\frac{32Ld\Delta}{3}
+u_t^2L^2d^2,
\end{aligned}
$$
where the third step follows from \eqref{eq:alg1:bound_Eg2_ver2} of Lemma~\ref{lemma:alg1:bounds_func}, and the fourth step utilizes the $L$-smoothness of $f_i$. Taking the total expectation completes the proof.
\end{proof}

By plugging \eqref{eq:alg1_graddom:gt_second_moment} into \eqref{eq:alg1:lemma_basic_ineq} of Lemma~\ref{lemma:alg1:basic_ineq} and using the fact that $f$ is $\mu$-gradient dominated, we can prove the following result.

\begin{lemma}\label{lemma:alg1_graddom:opt_gap_bound}
Suppose
$$
\eta_t L
\leq \frac{3\mu}{32L d}.
$$
Then for each $t\geq 1$,
\begin{equation}\label{eq:alg1_graddom:opt_gap_bound}
\mathbb{E}[\delta(t)]
\leq
\left(
1-\frac{\eta_t\mu}{2}
\right)
\mathbb{E}\!\left[\delta(t\!-\!1)\right]
+
\frac{3\eta_tL^2}{2n}e_{\mathrm{c}}(t\!-\!1)
+\frac{16\eta_t^2L^2d\Delta}{3}
+2\eta_tu_t^2L^2d
\end{equation}
\end{lemma}
\begin{proof}
Plugging \eqref{eq:alg1_graddom:gt_second_moment} into \eqref{eq:alg1:lemma_basic_ineq} and taking the total expectation yield
$$
\begin{aligned}
\mathbb{E}\!\left[
\delta(t)\right]
\leq\ &
\mathbb{E}\!\left[
\delta(t-1)
\right]
-\frac{\eta_t}{2}
\mathbb{E}\!\left[
\left\|\nabla f(\bar x(t-1))\right\|^2
\right]
+
\frac{\eta_tL^2}{n}\left(1
\!+\!4\eta_t Ld\right)
e_{\mathrm{c}}(t-1) \\
&
+\frac{16\eta_t^2 L^2d}{3}
\mathbb{E}\!\left[
\delta(t-1)
\right]
+\frac{16\eta_t^2 L^2d\Delta}{3}
+\eta_tu_t^2L^2 
+\frac{\eta_t^2L}{2}\cdot u_t^2L^2d^2
\\
\leq\ &
\!\left(
\!1+\frac{16\eta_t^2L^2d}{3}
\right)
\mathbb{E}\!\left[\delta(t\!-\!1)\right]
-\frac{\eta_t}{2}\mathbb{E}\!\left[
\left\|\nabla f(\bar x(t\!-\!1))\right\|^2
\right]
+
\frac{3\eta_tL^2}{2n}e_{\mathrm{c}}(t\!-\!1)
+\frac{16\eta_t^2L^2d\Delta}{3}
+2\eta_tu_t^2L^2d.
\end{aligned}
$$
where we used $\eta_t\leq
\mfrac{3\mu}{32L^2d}
\leq\mfrac{1}{8Ld}
$. The bound \eqref{eq:alg1_graddom:opt_gap_bound} then follows by using $2\mu\delta(t\!-\!1)\leq\|\nabla f(\bar x(t\!-\!1))\|^2$ and again $\eta_t\leq
\mfrac{3\mu}{32L^2d}$.
\end{proof}

The following lemma gives a coarse estimate of the convergence rate of $\mathbb{E}[\delta(t)]$, which will be refined later.
\begin{lemma}\label{lemma:alg1_graddom:boundedness_delta}
Suppose $\eta_t$ satisfies the conditions of Theorem~\ref{theorem:alg1_grad_dom}. Then $\mathbb{E}[\delta(t)]= O(t^{-1/2})$.
\end{lemma}
\begin{proof}
It can be checked that the conditions of both Lemma~\ref{lemma:alg1_graddom:consns_err_bound} and Lemma~\ref{lemma:alg1_graddom:opt_gap_bound} are satisfied by the choice of $\eta_t$ in Theorem~\ref{theorem:alg1_grad_dom}. By \eqref{eq:alg1_graddom:consns_err_one_step} and \eqref{eq:alg1_graddom:opt_gap_bound}, we have
$$
\begin{bmatrix}
e_{\mathrm{c}}(t)/n \\
\mathbb{E}[\delta(t)]
\end{bmatrix}
\leq
\begin{bmatrix}
\mfrac{1\!+\!\rho^2}{2}
&
\rho^2
\left(\mfrac{4d}{3}+\mfrac{6\rho^2}{1\!\!-\rho^2}\right)\eta_t^2 \\
\mfrac{3L^2}{2}\eta_t &
1-\mfrac{\eta_t\mu}{2}
\end{bmatrix}
\begin{bmatrix}
e_{\mathrm{c}}(t\!-\!1)/n \\
\mathbb{E}[\delta(t\!-\!1)]
\end{bmatrix}
+
\upsilon_t,
$$
where
$$
\upsilon_t=
\begin{bmatrix}
4\eta_t^2\rho^2L\Delta
\left(\mfrac{4d}{3}+\mfrac{6\rho^2}{1\!-\!\rho^2}\right)
+\eta_t^2u_t^2\rho^2L^2
\left(d^2+\mfrac{6\rho^2}{1\!-\!\rho^2}\right)
 \\
\mfrac{16\eta_t^2 L^2\Delta d}{3}
+2\eta_tu_t^2L^2d
\end{bmatrix}.
$$
By using $\eta_t=\mfrac{2\alpha_\eta}{\mu(t+t_0)}$ and $u_t=O(1/\sqrt{t})$, it can be shown by straightforward calculation that
$$
\left\|
\begin{bmatrix}
\mfrac{1\!+\!\rho^2}{2}
&
\rho^2
\left(\mfrac{4d}{3}+\mfrac{6\rho^2}{1\!\!-\rho^2}\right)\eta_t^2 \\
\mfrac{3L^2}{2}\eta_t &
1-\mfrac{\eta_t\mu}{2}
\end{bmatrix}
\right\|
=
1-\frac{\alpha_\eta}{t}+O(t^{-2}).
$$
Therefore there exists $T\geq 1$ such that
$$
\left\|
\begin{bmatrix}
\mfrac{1\!+\!\rho^2}{2}
&
\rho^2
\left(\mfrac{4d}{3}+\mfrac{6\rho^2}{1\!\!-\rho^2}\right)\eta_t^2 \\
\mfrac{3L^2}{2}\eta_t &
1-\mfrac{\eta_t\mu}{2}
\end{bmatrix}
\right\|
\leq 1-\frac{\alpha_\eta}{2t}
$$
for all $t\geq T$. Therefore
$$
\left\|
\begin{bmatrix}
e_{\mathrm{c}}(t)/n \\
\mathbb{E}[\delta(t)]
\end{bmatrix}\right\|
\leq
\left(1-\frac{\alpha_\eta}{2t}\right)
\left\|
\begin{bmatrix}
e_{\mathrm{c}}(t\!-\!1)/n \\
\mathbb{E}[\delta(t\!-\!1)]
\end{bmatrix}\right\|
+\|\upsilon_t\|,
\qquad\forall t\geq T,
$$
and by induction, we see that
$$
\left\|
\begin{bmatrix}
e_{\mathrm{c}}(t)/n \\
\mathbb{E}[\delta(t)]
\end{bmatrix}\right\|
\leq
\left\|
\begin{bmatrix}
e_{\mathrm{c}}(T\!-\!1)/n \\
\mathbb{E}[\delta(T\!-\!1)]
\end{bmatrix}\right\|
\cdot\prod_{\tau=T}^t
\left(1-\frac{\alpha_\eta}{2\tau}\right)
+\sum_{\tau=T}^t
\|\upsilon_\tau\|\prod_{s=\tau+1}^t
\left(1-\frac{\alpha_\eta}{2s}\right).
$$
Since for any $T\leq t_1\leq t_2+1$, we have
$$
\prod_{s=t_1}^{t_2}
\left(
1-\frac{\alpha_\eta}{2s}\right)
\leq
\exp\left(-\sum_{s=t_1}^{t_2}\frac{\alpha_\eta}{2s}\right)
\leq
\exp\left(-\frac{\alpha_\eta}{2}
\left(\ln(t_2+1)-\ln(t_1)\right)\right)
=\left(\frac{t_1}{t_2+1}\right)^{\alpha_\eta/2},
$$
we can see that
\begin{equation}\label{eq:alg1_graddom:boundedness_delta_temp}
\begin{aligned}
\left\|
\begin{bmatrix}
e_{\mathrm{c}}(t)/n \\
\mathbb{E}[\delta(t)]
\end{bmatrix}\right\|
\leq
\left\|
\begin{bmatrix}
e_{\mathrm{c}}(T\!-\!1)/n \\
\mathbb{E}[\delta(T\!-\!1)]
\end{bmatrix}\right\|
\left(\frac{T}{t+1}\right)^{\alpha_\eta/2}
+\sum_{\tau=T}^t
\|\upsilon_\tau\|
\left(\frac{\tau+1}{t+1}\right)^{\alpha_\eta/2}.
\end{aligned}
\end{equation}
Finally, by noticing that $\alpha_\eta>1$, $\|\upsilon_t\|=O(1/t^2)$ and that
$$
\begin{aligned}
\sum_{\tau=T}^t
\frac{1}{\tau^2}
\left(\frac{\tau+1}{t+1}\right)^{\alpha_\eta/2}
\leq\ &
\frac{1}{(t+1)^{\alpha_\eta/2}}
\cdot\frac{(T+1)^2}{T^2}
\sum_{\tau=T}^t
(\tau+1)^{\alpha_\eta/2-2}
=
\left\{
\begin{aligned}
& O(1/t), & \ \ & \alpha_\eta>2, \\
& O(\ln t/t), &\ \ & \alpha_\eta=2, \\
& O(t^{-\alpha_\eta/2}), &\ \ & 1<\alpha_\eta<2
\end{aligned}
\right. \\
=\ &
O(t^{-1/2}),
\end{aligned}
$$
we can see that $\mathbb{E}[\delta(t)]=O(t^{-1/2})$.
\end{proof}

We are now ready to prove Theorem~2.

\begin{proof}[Proof of Theorem~2]
By Lemma~\ref{lemma:alg1_graddom:boundedness_delta}and \eqref{eq:conv_exp_poly_asymp}, we can see that
$$
\sum_{\tau=1}^t
\frac{\mathbb{E}[\delta(\tau-1)]}{(\tau+t_0)^2}
\left(\frac{1+\rho^2}{2}\right)^{t-\tau}
=O\!\left(\frac{1}{t^{2+1/2}}\right).
$$
Therefore from \eqref{eq:alg1_graddom:consns_err_bound} of Lemma~\ref{lemma:alg1_graddom:consns_err_bound}, we see that
\begin{equation}\label{eq:alg1_graddom:consns_err_bound_temp1}
\begin{aligned}
\frac{e_{\mathrm{c}}(t)}{n}
\leq\ &
\frac{32\alpha_\eta^2
\rho^2 L\Delta}{\mu^2(1-\rho^2)}\left(\frac{4d}{3}
+\frac{6\rho^2}{1\!-\!\rho^2}\right)
\frac{1}{t^2}
+ o(t^{-2}),
\end{aligned}
\end{equation}
which is just the bound \eqref{eq:alg1_consensus_grad_dom}. On the other hand, by Lemma~\ref{lemma:alg1_graddom:opt_gap_bound} and mathematical induction, we get
$$
\begin{aligned}
\mathbb{E}[\delta(t)]
\leq\ &
\delta(0)
\prod_{\tau=1}^t
\left(
1-\frac{\eta_\tau\mu}{2}\right)
+
\frac{3L^2}{2}
\sum_{\tau=1}^t
\eta_\tau
\frac{e_{\mathrm{c}}(\tau\!-\!1)}{n}
\prod_{s=\tau+1}^t
\left(
1-\frac{\eta_s\mu}{2}\right) \\
&
+\frac{16L^2\Delta d}{3}
\sum_{\tau=1}^t
\eta_\tau^2\prod_{s=\tau+1}^t
\left(
1-\frac{\eta_s\mu}{2}\right)
+
2L^2d
\sum_{\tau=1}^t
\eta_\tau u_\tau^2
\left(
1-\frac{\eta_s\mu}{2}\right).
\end{aligned}
$$
Since for any $t_1\leq t_2+1$, we have
$$
\begin{aligned}
\prod_{s=t_1}^{t_2}
\left(
1-\frac{\eta_s\mu}{2}\right)
\leq\ &
\exp\left(-\sum_{s=t_1}^{t_2}\frac{\eta_s\mu}{2}\right)
=
\exp\left(-\alpha_\eta\sum_{s=t_1}^{t_2}\frac{1}{s+t_0}\right) \\
\leq\ &
\exp\left(-\alpha_\eta
\left(\ln(t_2+t_0+1)-\ln(t_1+t_0)\right)\right)
=\left(\frac{t_1+t_0}{t_2+t_0+1}\right)^{\alpha_\eta},
\end{aligned}
$$
by plugging in the conditions on $\eta_t$ and $u_t$, we get
$$
\begin{aligned}
\mathbb{E}[\delta(t)]
\leq\ &
\delta(0)
\left(\frac{t_0+1}{t+t_0+1}\right)^{\alpha_\eta}
+\frac{3\alpha_\eta L^2}{\mu}
\sum_{\tau=1}^t\frac{ e_{\mathrm{c}}(\tau\!-\!1)}{n(\tau+t_0)}
\left(\frac{\tau+t_0+1}{t+t_0+1}\right)^{\alpha_\eta} \\
&
+
\left(\frac{64\alpha_\eta^2 L^2\Delta d}{3\mu^2}
+\frac{4\alpha_\eta\alpha_u^2 L^2d}{\mu}\right)
\sum_{\tau=1}^t
\frac{1}{(\tau+t_0)^2}
\left(\frac{\tau+t_0+1}{t+t_0+1}\right)^{\alpha_\eta}.
\end{aligned}
$$
By \eqref{eq:alg1_graddom:consns_err_bound_temp1}, we see that
$$
\begin{aligned}
& \sum_{\tau=1}^t
\frac{e_{\mathrm{c}}(\tau\!-\!1)}{n(\tau+t_0)}\left(\frac{\tau+t_0+1}{t+t_0+1}\right)^{\alpha_\eta} \\
\leq\ &
C_1\sum_{\tau=1}^t\frac{1}{\tau^2(\tau+t_0)}\left(\frac{\tau+t_0+1}{t+t_0+1}\right)^{\alpha_\eta}
=
\frac{C_1}{(t+t_0+1)^{\alpha_\eta}}
\cdot\left\{
\begin{aligned}
&O(t^{\alpha_\eta-2}), &\ \ & \alpha_\eta>2, \\
& O(\ln t), &\ \ & \alpha_\eta = 2, \\
& C_2, &\ \ & 1<\alpha_\eta<2
\end{aligned}
\right. \\
=\ &
o(t^{-1}),
\end{aligned}
$$
where $C_1$ and $C_2$ are some positive constant. On the other hand,
$$
\begin{aligned}
\sum_{\tau=1}^t
\frac{1}{(\tau+t_0)^2}
\left(\frac{\tau+t_0+1}{t+t_0+1}\right)^{\alpha_\eta}
\leq\ &
\frac{1}{(t+t_0+1)^{\alpha_\eta}}
\left(
\frac{t_0+2}{t_0+1}
\right)^2
\sum_{\tau=1}^t
(\tau+t_0+1)^{\alpha_\eta} \\
=\ &
\left(
\frac{t_0+2}{t_0+1}
\right)^2
\cdot\frac{1}{t}
+o(t^{-1})
\leq
\frac{3}{2}
\cdot\frac{1}{t}
+o(t^{-1}),
\end{aligned}
$$
where we used the fact that
$$
\left(
\frac{t_0+2}{t_0+1}
\right)^2
\leq \left(
1+\frac{3\mu^2}{64\alpha_\eta L^2 d}
\right)^2
\leq 
\left(
1+\frac{3}{64}
\right)^2
\leq\frac{3}{2}.
$$
Therefore we obtain
\begin{equation}\label{eq:alg1_graddom:delta_bound_final}
\mathbb{E}[\delta(t)]
\leq
\left(\frac{32\alpha_\eta^2 L^2\Delta d}{\mu^2}
+\frac{6\alpha_\eta\alpha_u^2 L^2d}{\mu}\right)\frac{1}{t}
+o(t^{-1}).
\end{equation}
\end{proof}

\section{Proof of Theorem~\ref{theorem:alg2_non_grad_dom}}

We first bound the error of the $2d$-point gradient estimator.
\begin{lemma}\label{lemma:alg2:finite_difference}
Let $f:\mathbb{R}^d\rightarrow\mathbb{R}$ be $L$-smooth. Then for any $x\in\mathbb{R}^d$,
$$
\left\|
\sum_{k=1}^d\frac{f(x+ue_k)-f(x-ue_k)}{2u}e_k-\nabla f(x)\right\|
\leq \frac{1}{2}uL\sqrt{d}.
$$
\end{lemma}
\begin{proof}
We have
$$
\begin{aligned}
\left\|
\sum_{k=1}^d\frac{f(x+ue_k)-f(x-ue_k)}{2u}e_k-\nabla f(x)\right\|
=\ &\left\|
\sum_{k=1}^d\left(\frac{f(x+ue_k)-f(x-ue_k)}{2u}-\langle\nabla f(x),e_k\rangle\right)e_k\right\| \\
=\ &
\left(
\sum_{k=1}^d
\left|
\frac{f(x+ue_k)-f(x-ue_k)}{2u}-\langle\nabla f(x),e_k\rangle
\right|^2
\right)^{1/2} \\
\leq\ &
\left(
\sum_{k=1}^d \left(
\frac{1}{2} uL
\right)^2
\right)^{1/2}=\frac{1}{2}uL\sqrt{d},
\end{aligned}
$$
where we used \eqref{eq:alg1:finite_difference} of Lemma \ref{lemma:alg1:bounds_func}.
\end{proof}

We shall use the notations
$$
x(t)=\begin{bmatrix}
x^1(t) \\
\vdots \\
x^n(t)
\end{bmatrix},
\quad
g(t)=\begin{bmatrix}
g^1(t) \\
\vdots \\
g^n(t)
\end{bmatrix},
\quad
s(t)=\begin{bmatrix}
s^1(t) \\
\vdots \\
s^n(t)
\end{bmatrix},
\qquad
\bar x(t)=\frac{1}{n}\sum_{i=1}^n x^i(t),
\quad
\bar g(t)=\frac{1}{n}\sum_{i=1}^n g^i(t),
$$
and $\delta(t)=f(\bar x(t))-f^\ast$, $e_{\mathrm{c}}(t)=\|x(t)-\mathbf{1}_n\otimes\bar{x}(t)\|^2$, $e_{\mathrm{g}}(t)= \|s(t)-\mathbf{1}_n \otimes \bar g(t)\|^2$. It's not hard to see that the iterations of Algorithm~\ref{alg:multipoint} can be equivalently written as
$$
\begin{aligned}
s(t) &= (W\otimes I_d)(s(t-1)+g(t)-g(t-1)), \\
x(t) &=(W\otimes I_d)(x(t-1)-\eta s(t)).
\end{aligned}
$$
We also have
$$
\frac{1}{n}\sum_{i=1}^n s^i(t)=\bar g(t),\qquad\qquad
\bar x(t) = \bar x(t-1)-\eta\bar g(t).
$$

\begin{lemma}\label{lemma:alg2:basic_ineq}
Suppose $\eta L\leq 1/6$. Then
\begin{equation}\label{eq:lemma_basic_ineq}
\delta(t)
\leq
\delta(t-1)
-\frac{\eta}{3}\|\nabla f(\bar x(t-1))\|^2
+
\frac{4\eta L^2}{3n}
e_{\mathrm{c}}(t-1)
+\frac{\eta u_t^2 L^2 d}{3}.
\end{equation}
\end{lemma}
\begin{proof}
By $\bar x(t)=\bar x(t-1)-\eta\bar g(t)$ and the $L$-smoothness of the function $f$, we have
$$
\begin{aligned}
f(\bar x(t))
\leq\ &
f(\bar x(t-1))
-\eta \langle\nabla f(\bar x(t-1)),\bar g(t)\rangle
+\frac{\eta^2 L}{2}\|\bar g(t)\|^2 \\
=\ &
f(\bar x(t-1))
-\eta\|\nabla f(\bar x(t-1))\|^2
+\frac{\eta^2 L}{2}\|\bar g(t)\|^2 \\
&- \eta
\left\langle \nabla f(\bar x(t-1)),
\frac{1}{n}\sum_{i=1}^n(g^i(t)-f_i(\bar x(t-1))\right\rangle \\
\leq\ &
f(\bar x(t-1))
-\frac{\eta}{2}\|\nabla f(\bar x(t-1))\|^2
+\frac{\eta^2 L}{2}\|\bar g(t)\|^2
+\frac{\eta}{2}
\left\|\frac{1}{n}\sum_{i=1}^n(g^i(t)-\nabla f_i(\bar x(t-1))\right\|^2.
\end{aligned}
$$
Then, by Lemma \ref{lemma:alg2:finite_difference},
\begin{equation}\label{eq:alg2:temp_bound_1}
\begin{aligned}
&\left\|\frac{1}{n}\sum_{i=1}^n\left(g^i(t)-\nabla f_i(\bar x(t-1))\right)\right\|^2 \\
\leq \ &
2\left\|\frac{1}{n}\sum_{i=1}^n\left(\nabla f_i(x^i(t-1))-\nabla f_i(\bar x(t-1))\right)\right\|^2
+2\left(\frac{1}{n}\sum_{i=1}^n
\|g^i(t)-\nabla f_i(x^i(t-1))\|\right)^2 \\
\leq\ &
2\left(\frac{1}{n}\sum_{i=1}^n L
\|x^i(t-1)-\bar x(t-1)\|\right)^2
+\frac{1}{2}u_t^2L^2d \\
\leq\ &
\frac{2L^2}{n}\|x(t-1)-\mathbf{1}_n\otimes\bar x(t-1)\|^2
+\frac{1}{2}u_t^2L^2d,
\end{aligned}
\end{equation}
we see that
$$
\begin{aligned}
f(\bar x(t))
\leq\ &
f(\bar x(t-1))
-\frac{\eta}{2}\|\nabla f(\bar x(t-1))\|^2
+\frac{\eta^2 L}{2}\|\bar g(t)\|^2 \\
&
+
\frac{\eta L^2}{n}\|x(t-1)-\mathbf{1}_n\otimes\bar x(t-1)\|^2
+\frac{\eta u_t^2 L^2 d}{4}.
\end{aligned}
$$
Next, we bound the term $\|\bar g(t)\|^2$:
$$
\begin{aligned}
\|\bar g(t)\|^2
=\ &
\left\|\frac{1}{n}\sum_{i=1}^n g^i(t)\right\|^2
\leq
2\left\|\nabla f(\bar x(t-1))\right\|^2
+2\left\|
\frac{1}{n}\sum_{i=1}^n(g^i(t)
-\nabla f_i(\bar x(t-1)))\right\|^2 \\
\leq\ &
2\left\|\nabla f(\bar x(t-1))\right\|^2
+\frac{4L^2}{n}
\|x(t)-\mathbf{1}_n\otimes\bar x(t)\|^2
+u_t^2L^2d.
\end{aligned}
$$
Then we see that
\begin{equation}\label{eq:alg2:temp_basic_ineq}
\begin{aligned}
f(\bar x(t))
\leq\ &
f(\bar x(t-1))
-\frac{\eta}{2}
\left(1-2\eta L\right)\|\nabla f(\bar x(t-1))\|^2\\
&
+
\frac{\eta L^2}{n}
\left(1+2\eta L\right)\|x(t)-\mathbf{1}_n\otimes\bar x(t-1)\|^2
+\frac{\eta u_t^2 L^2 d}{4}
\left(1+2\eta L\right).
\end{aligned}
\end{equation}
Finally, by using $\eta L\leq 1/6$, we get the desired result.
\end{proof}

\begin{lemma}\label{lemma:alg2:d0_bound}
We have
$$
e_{\mathrm{g}}(1)=\|s(1)-\mathbf{1}_n\otimes\bar g(1)\|^2
\leq \rho^2\left(
\frac{3}{2}\sum_{i=1}^n \|\nabla f_i(x^i(0))\|^2+\frac{3}{4}nu_1^2L^2d
\right).
$$
\end{lemma}
\begin{proof}
Since $s(0)=g(0)=0$, we have
$$
\|s(1)-\mathbf{1}_n\otimes\bar g(1)\|^2
=\|(W\otimes I_d)(g(1)-\mathbf{1}_n\otimes\bar g(1))\|^2
\leq \rho^2
\|g(1)-\mathbf{1}_n\otimes\bar g(1)\|^2.
$$
Then since
$$
\begin{aligned}
\|g(1)-\mathbf{1}_n\otimes\bar g(1)\|^2
=\ &
\|g(1)\|^2
+n\|\bar g(1)\|^2
-2\sum_{i=1}^n\left\langle g^i(1),\frac{1}{n}\sum_{j=1}^n g^j(1)\right\rangle \\
=\ &
\|g(1)\|^2
-n\|\bar g(1)\|^2
\leq \|g(1)\|^2,
\end{aligned}
$$
and by Lemma \ref{lemma:alg2:finite_difference},
$$
\begin{aligned}
\|g(1)\|^2
\leq\ &
\sum_{i=1}^n
\left[\frac{3}{2}\|\nabla f^i(x(0))\|^2
+3\|g^i(1)-\nabla f^i(x(0))\|^2
\right] \\
\leq\ &
\frac{3}{2}
\sum_{i=1}^n\|\nabla f^i(x(0))\|^2
+3\sum_{i=1}^n\left(\frac{1}{2}u_1L\sqrt{d}\right)^2
= \frac{3}{2}
\sum_{i=1}^n\|\nabla f^i(x(0))\|^2
+\frac{3}{4}nu_1^2L^2d,
\end{aligned}
$$
we get the desired result.
\end{proof}

The following lemma characterizes the consensus procedure.
\begin{lemma}\label{lemma:alg2_consensus_basic}
Suppose $\eta L\leq 1/6$. Then we have the following component-wise inequality
\begin{equation}\label{eq:temp_consensus_alg2}
\begin{aligned}
\begin{bmatrix}
\frac{5\eta}{2\sqrt{57}L} e_{\mathrm{g}}(t) \\
e_{\mathrm{c}}(t\!-\!1)
\end{bmatrix}
\leq
A
\begin{bmatrix}
\frac{5\eta}{2\sqrt{57}L} e_{\mathrm{g}}(t\!-\!1) \\
\mathrm{e}_{\mathrm{c}}(t\!-\!2)
\end{bmatrix}
+
\frac{2n\eta^3 L \rho^2(1\!+\!2\rho^2)}{3(1-\rho^2)}
\begin{bmatrix}
2\|\nabla f(\bar x(t \!-\! 2))\|^2
\!+\!\mfrac{5}{4}\mfrac{u_{t\!-\!1}^2d}{\eta^2} \\
0
\end{bmatrix},
\end{aligned}
\end{equation}
where
\begin{equation}\label{eq:alg2:mat_A}
A\coloneqq
\begin{bmatrix}
\frac{1+2\rho^2}{3}
+\frac{18\rho^4(1+2\rho^2)}{1-\rho^2}\eta^2 L^2
&
\frac{2\sqrt{57}\rho^2(1+2\rho^2)}{5(1-\rho^2)}\eta L
\\
\frac{2\sqrt{57}\rho^2(1+2\rho^2)}{5(1-\rho^2)}\eta L
&
\frac{1+2\rho^2}{3}
\end{bmatrix}
\end{equation}
\end{lemma}
\begin{proof}
We have
$$
\begin{aligned}
&s(t)-\mathbf{1}_n\otimes\bar g(t) \\
=\ &(W\otimes I_d)\left(s(t-1)-\mathbf{1}_n\otimes\bar g(t-1)+g(t)-g(t-1)-\mathbf{1}_n\otimes\bar g(t)+\mathbf{1}_n\otimes\bar g(t-1)\right).
\end{aligned}
$$
Then since
$$
\begin{aligned}
&\|g(t)-g(t-1)-\mathbf{1}_n\otimes\bar g(t)+\mathbf{1}_n\otimes\bar g(t-1)\|^2 \\
=\ & \|g(t)-g(t-1)\|^2
+ n \|\bar g(t)-\bar g(t-1)\|^2
-2\sum_{i=1}^n \langle g^i(t)-g^i(t-1), \bar g(t)-\bar g(t-1)\rangle \\
=\ & \|g(t)-g(t-1)\|^2
- n \|\bar g(t)-\bar g(t-1)\|^2
\leq \|g(t)-g(t-1)\|^2,
\end{aligned}
$$
we have
$$
\begin{aligned}
&
e_{\mathrm{g}}(t)=\|s(t)-\mathbf{1}_n\otimes \bar g(t)\|^2 \\
\leq\ &
\rho^2\left(\|s(t-1)-\mathbf{1}_n\otimes \bar g(t-1)\|
+\|g(t)-g(t-1)\|\right)^2 \\
\leq\ &
\rho^2
\cdot\left[
\left(1+\frac{1-\rho^2}{3\rho^2}\right)\|s(t-1)-\mathbf{1}_n\otimes \bar g(t-1)\|^2
+\left(1+\frac{3\rho^2}{1-\rho^2}\right)\|g(t)-g(t-1)\|^2
\right] \\
=\ &
\frac{1+2\rho^2}{3}\|s(t-1)-\mathbf{1}_n\otimes \bar g(t-1)\|^2
+ \frac{\rho^2(1+2\rho^2)}{1-\rho^2}\sum_{i=1}^n \|g^i(t)-g^i(t-1)\|^2.
\end{aligned}
$$
Now since
$$
\begin{aligned}
\|g^i(t)-g^i(t-1)\|^2 \leq\ &
2\left\|
\nabla f_i(x^i(t-1))-\nabla f_i(x^i(t-2))
\right\|^2 \\
& + 
2 \left\|
g^i(t)-\nabla f_i(x^i(t-1))
-g^i(t-1)+\nabla f_i(x^i(t-2))
\right\|^2
\\
\leq\ &
2\left\|
\nabla f_i(x^i(t-1))-\nabla f_i(x^i(t-2))
\right\|^2
+2\left(\frac{u_t+u_{t-1}}{2}L\sqrt{d}\right)^2 \\
\leq\ &
2L^2 \|x^i(t-1)-x^i(t-2)\|^2
+2u_{t-1}^2L^2d,
\end{aligned}
$$
we get
$$
\begin{aligned}
e_{\mathrm{g}}(t)
\leq
\frac{1+2\rho^2}{3}
e_{\mathrm{g}}(t-1)+
\frac{2\rho^2(1+2\rho^2)L^2}{1-\rho^2} \|x(t-1)-x(t-2)\|^2
+\frac{2\rho^2(1+2\rho^2)L^2}{1-\rho^2}nu_{t-1}^2d.
\end{aligned}
$$
Now,
$$
\begin{aligned}
&x(t-1)-x(t-2) \\
=\ &((W\otimes I_d)-I_{nd})x(t-2)
-\eta (W\otimes I_d)s(t-1) \\
=\ &((W\otimes I_d)-I_{nd})(x(t-2)-\mathbf{1}_n\otimes\bar x(t-2))
-\eta (W\otimes I_d)(s(t-1)-\mathbf{1}_n\otimes\bar g(t-1)) \\
& - \eta \mathbf{1}_n\otimes
(\bar g(t-1)-\nabla f(\bar x(t-2)))
-\eta\mathbf{1}_n\otimes \nabla f(\bar x(t-2)).
\end{aligned}
$$
We notice that for any $u_1,\ldots,u_n\in\mathbb{R}^{d}$ and $v\in\mathbb{R}^d$, we have
\begin{equation}\label{eq:temp_inner_prod_vanish}
\sum_{i=1}^n\left\langle u_i-\frac{1}{n}\sum_{j=1}^n u_j,
v\right\rangle=0,
\qquad
\textrm{and}
\qquad
\sum_{i=1}^n\left\langle \sum_{j=1}^n W_{ij}u_j-\frac{1}{n}\sum_{j=1}^n u_j,
v\right\rangle=0.
\end{equation}
In addition, similar to \eqref{eq:alg2:temp_bound_1}, we can show that
$$
\begin{aligned}
\|\bar g(t-1)-\nabla f(\bar x(t-2))\|^2 
\leq\ &
\frac{18L^2}{17n}\|x(t-2)-\mathbf{1}_n\otimes\bar x(t-2)\|^2
+\frac{9}{2}u_{t-1}^2L^2d.
\end{aligned}
$$
Therefore we get
$$
\begin{aligned}
&\|x(t-1)-x(t-2)\|^2 \\
=\ &
\|((W\otimes I_d)-I_{nd})(x(t-2)-\mathbf{1}_n\otimes\bar x(t-2))-\eta(W\otimes I_d)(s(t-1)-\mathbf{1}_n\otimes\bar g(t-1))\|^2 \\
&+\eta^2 n\|\bar g(t-1)-\nabla f(\bar x(t-2))+\nabla f(\bar x(t-2))\|^2 \\
\leq\ &
\frac{9}{2}\|x(t-2)-\mathbf{1}_n\otimes\bar x(t-2)\|^2
+9\eta^2\rho^2
\|s(t-1)-\mathbf{1}_n\otimes\bar g(t-1)\|^2 \\
&
+2\eta^2 n\|\bar g(t-1)
-\nabla f(\bar x(t-2))\|^2
+2\eta^2 n\|\nabla f(\bar x(t-2))\|^2
\\
\leq\ &
\left(\frac{9}{2}+\frac{36}{17}\eta^2 L^2\right)\|x(t-2)-\mathbf{1}_n\otimes\bar x(t-2)\|^2 \\
&+9\eta^2\rho^2
\|s(t-1)-\mathbf{1}_n\otimes\bar g(t-1)\|^2
+
2\eta^2 n\|\nabla f(\bar x(t-2))\|^2
+9\eta^2nu_{t-1}^2L^2d
\\
\leq\ &
\frac{155}{34}\|x(t-2)-\mathbf{1}_n\otimes\bar x(t-2)\|^2 
+9\eta^2\rho^2
\|s(t-1)-\mathbf{1}_n\otimes\bar g(t-1)\|^2 \\
&+
2\eta^2 n\|\nabla f(\bar x(t-2))\|^2
+\frac{1}{4}nu_{t-1}^2d,
\end{aligned}
$$
where the first inequality follows from $\|W\otimes I_d-I_{nd}\|\leq 2$ and that $\|u+v\|^2\leq (1+1/\epsilon)\|u\|^2+(1+\epsilon)\|v\|^2$ for any vectors $u,v$ and $\epsilon>0$, and the third inequality follows from $\eta L\leq 1/6$. Consequently
$$
\begin{aligned}
e_{\mathrm{g}}(t)
\leq\ &
\left(\frac{1+2\rho^2}{3}+\frac{18\rho^4(1+2\rho^2)}{1-\rho^2}\eta^2L^2\right)
e_{\mathrm{g}}(t-1)
+\frac{228\rho^2(1+2\rho^2)}{25(1-\rho^2)}
L^2
\|x(t-2)-\mathbf{1}_n\otimes\bar x(t-2)\|^2 \\
&+\frac{2\rho^2(1+2\rho^2)}{1-\rho^2}
\left(2\eta^2 L^2 n\|\nabla f(\bar x(t-2))\|^2
+\frac{5}{4}nL^2 u_{t-1}^2d\right),
\end{aligned}
$$
where we used $155/34<114/25$. On the other hand,
$$
\begin{aligned}
&
e_{\mathrm{c}}(t-1)=\|x(t-1)-\mathbf{1}_n\otimes \bar x(t-1)\|^2 \\
=\ &
\|(W\otimes I_d)[x(t-2)-\mathbf{1}_n\otimes\bar x(t-2)
-\eta(s(t-1)-\mathbf{1}_n\otimes\bar g(t-1))]\|^2 \\
\leq\ &
\frac{1+2\rho^2}{3} e_{\mathrm{c}}(t-2)
+\frac{\rho^2(1+2\rho^2)}{1-\rho^2}\eta^2
e_{\mathrm{g}}(t-1).
\end{aligned}
$$
By combining these results, we get the desired inequality \eqref{eq:temp_consensus_alg2}.
\end{proof}

\begin{lemma}
If
\begin{equation}\label{eq:cond_stepsize_1}
\eta L \leq 
\min\left\{
\frac{1}{6},
\frac{(1-\rho^2)^2}{4\rho^2(3+4\rho^2)}
\right\},
\end{equation}
then
\begin{equation}\label{eq:alg2:consensus_bound}
\begin{aligned}
\max\left\{
e_{\mathrm{c}}(t),
\frac{3\eta }{10L}
e_{\mathrm{g}}(t+1)
\right\}
\leq\ &
\!\left(
\frac{2+\rho^2}{3}
\right)^{\!t}
n R_0
+\frac{4n\eta^3 L \rho^2(1+2\rho^2)}{3(1-\rho^2)}
\sum_{\tau=0}^{t-1}
\left(\frac{2+\rho^2}{3}\right)^{\!\tau}
\!\|\nabla f(\bar x(t \!-\! \tau \!-\! 1))\|^2 \\
&
+\frac{5n\eta L d\rho^2(1+2\rho^2)}{6(1-\rho^2)}
\sum_{\tau=0}^{t-1} \left(\frac{2+\rho^2}{3}\right)^{\tau}u_{t-\tau}^2,
\end{aligned}
\end{equation}
where we recall that
$$
R_0
\coloneqq
\frac{1}{n}\sum_{i=1}^n\!
\left(\! \frac{\eta \rho^2 }{2L}\|\nabla f_i(x^i(0))\|^2
\!+\!\|x^i(0)-\bar x(0)\|^2 \!\right)
+\frac{\eta\rho^2 u_1^2 Ld}{4}.
$$
Consequently
\begin{equation}\label{eq:alg2:sum_consensus_error}
\begin{aligned}
\max\left\{\sum_{\tau=0}^{t-1}
e_{\mathrm{c}}(\tau),
\frac{3\eta}{10L}\sum_{\tau=1}^{t}
e_{\mathrm{g}}(\tau)\right\}
\leq\ &
\frac{3n R_0}{1-\rho^2}
+\frac{4n\eta^3L\rho^2(1+2\rho^2)}{(1-\rho^2)^2}
\sum_{\tau=0}^{t-2}\|\nabla f(\bar x(\tau))\|^2  \\
&
+\frac{5n \eta L d \rho^2(1+2\rho^2)}{2(1-\rho^2)^2}
\sum_{\tau=1}^{t-1}u_\tau^2.
\end{aligned}
\end{equation}
\end{lemma}
\begin{proof}

By induction on \eqref{eq:temp_consensus_alg2}, we get
$$
\begin{aligned}
\begin{bmatrix}
\frac{5\eta}{2\sqrt{57}L} e_{\mathrm{g}}(t+1) \\
e_{\mathrm{c}}(t)
\end{bmatrix}
\leq\ &
A^t \!
\begin{bmatrix}
\frac{5\eta}{2\sqrt{57}L}
e_{\mathrm{g}}(1) \\
e_{\mathrm{c}}(0)
\end{bmatrix} \\
&
+\frac{5n\eta^3 L\rho^2(1+2\rho^2)}{\sqrt{57}(1-\rho^2)}
\sum_{\tau=0}^{t-1}
A^{\tau} \!
\begin{bmatrix}
2\|\nabla f(\bar x(t\!-\!\tau\!-\!1))\|^2
+\frac{5}{4}\eta^{-2}u^2_{t-\tau}d \\
0
\end{bmatrix},
\end{aligned}
$$
and consequently
$$
\begin{aligned}
&
\max\left\{e_{\mathrm{c}}(t),\frac{5\eta}{2\sqrt{57}L}e_{\mathrm{g}}(t+1)\right\} \\
\leq\ &
\|A\|^t
\left(\frac{5\eta}{2\sqrt{57}L} e_{\mathrm{g}}(1)
+e_{\mathrm{c}}(0)\right) \\
&
+\frac{10n\eta^3 L\rho^2(1+2\rho^2)}{\sqrt{57}(1-\rho^2)}
\sum_{\tau=0}^{t-1} \|A\|^{\tau}\|\nabla f(\bar x(t-\tau-1))\|^2
+\frac{25n \eta L d\rho^2(1+2\rho^2)}{4\sqrt{57}(1-\rho^2)}
\sum_{\tau=0}^{t-1} \|A\|^{\tau}u_{t-\tau}^2,
\end{aligned}
$$
where we used the fact that $\max\{a,b\}\leq \sqrt{a^2+b^2}\leq a+b$ for any $a\geq 0$ and $b\geq 0$.

Now, since $A$ is symmetric, straightforward calculation shows that
\begin{equation}\label{eq:alg2:norm_A}
\|A\|
=\frac{1+2\rho^2}{3(1-\rho^2)}
\left(
1-\rho^2+27\rho^4(\eta L)^2
+\frac{3\sqrt{3}}{5}\sqrt{76\rho^4(\eta L)^2+675\rho^8(\eta L)^4}
\right).
\end{equation}
By solving the inequality $\|A\|\leq (2+\rho^2)/3$, we get
$$
(\eta L)^2
\leq \frac{25(1-\rho^2)^4}{\rho^4
(3402+8208\rho^2+4158\rho^4+2700\rho^6)}.
$$
It can be checked that
$$
\frac{1}{25}(3402+8208\rho^2+4158\rho^4+2700\rho^6)
\leq \left[4(3+4\rho^2)\right]^2,\qquad\forall \rho\in[0,1).
$$
Therefore if $\eta L$ satisfies \eqref{eq:cond_stepsize_1}, we have $\|A\|\leq(2+\rho^2)/3$. By Lemma \ref{lemma:alg2:d0_bound} and that $3/10<5/(2\sqrt{57})<1/3$, we get \eqref{eq:alg2:consensus_bound}. The bound \eqref{eq:alg2:sum_consensus_error} follows by taking the sum of \eqref{eq:alg2:consensus_bound} and using
$$
\sum_{\tau=1}^{t-1}\sum_{s=0}^{\tau-1}\theta^s a_{\tau-s}
=\sum_{\tau=1}^{t-1}\sum_{s=1}^{\tau}\theta^{\tau-s} a_{s}
=\sum_{s=1}^{t-1}a_{s}\sum_{\tau=s}^{t-1}\theta^{\tau-s}
\leq \frac{1}{1-\theta}\sum_{s=1}^{t-1}a_{s}
$$
for any nonnegative sequence $(a_s)_{s\in\mathbb{N}}$ and $\theta\in(0,1)$.
\end{proof}

Now we are ready to prove Theorem 3 in the main text.

\begin{proof}[Proof of Theorem 3]
Let $t\geq 1$ be arbitrary. By Lemma \ref{lemma:alg2:basic_ineq} and \eqref{eq:alg2:sum_consensus_error}, we see that
$$
\begin{aligned}
0\leq\ &
\delta(0)
-\frac{\eta}{3}\sum_{\tau=0}^{t-1}\|\nabla f(\bar x(\tau))\|^2
+\frac{\eta L^2d}{3}\sum_{\tau=1}^t u_\tau^2\\
&
+\frac{4\eta L^2}{3n}
\left(
\frac{3nR_0}{1-\rho^2}
+\frac{4n\eta^3 L\rho^2(1+2\rho^2)}{(1-\rho^2)^2}
\sum_{\tau=0}^{t-2}\|\nabla f(\bar x(\tau))\|^2 
+\frac{5n \eta L d\rho^2(1+2\rho^2)}{2(1-\rho^2)^2}
\sum_{\tau=1}^{t-1}u_\tau^2\right)  \\
\leq\ &
\delta(0)
+\frac{4\eta L^2R_0}{1-\rho^2}
-\frac{\eta}{3}
\left(
1
-\frac{16\eta^3 L^3 \rho^2 (1+2\rho^2)}{(1-\rho^2)^2}\right)\sum_{\tau=0}^{t-1}\|\nabla f(\bar x(\tau))\|^2 \\
&
+\left(\frac{10\eta L\rho^2(1+2\rho^2)}{3(1-\rho^2)^2}
+\frac{1}{3}\right)\eta L^2 d
\sum_{\tau=1}^{t}u_\tau^2.
\end{aligned}
$$
Then since
$$
\begin{aligned}
\frac{16\eta^3 L^3\rho^2 (1+2\rho^2)}{(1-\rho^2)^2}
\leq\ &
\frac{16}{36}
\cdot \frac{(1-\rho^2)^2}{4\rho^2(3+4\rho^2)}\frac{ \rho^2(1+2\rho^2)}{(1-\rho^2)^2}
=
\frac{4}{9}\frac{(1+2\rho^2)}{4(2+4\rho^2)}
=\frac{1}{18},
\end{aligned}
$$
and $\frac{1}{3}(1-1/18)=\frac{17}{54}\geq\frac{5}{16}$, and
$$
\frac{10\eta L\rho^2(1+2\rho^2)}{3(1-\rho^2)^2}
\leq \frac{10}{3}\cdot\frac{(1-\rho^2)^2}{4\rho^2(3+4\rho^2)}\cdot \frac{\rho^2(1+2\rho^2)}{(1-\rho^2)^2}
\leq\frac{5}{12},
$$
we get
$$
\begin{aligned}
0\leq\ &
f(\bar x(0))
+\frac{4\eta L^2 R_0}{1-\rho^2}
-\frac{5\eta}{16}
\sum_{\tau=0}^{t-1}\|\nabla f(\bar x(\tau))\|^2
+\frac{3}{4}\eta L^2 d
\sum_{\tau=1}^{t}u_\tau^2.
\end{aligned}
$$
Since $u_t^2$ is summable, this implies that $\|\nabla f(\bar x(t))\|$ converges to zero, and we have
\begin{equation}%\label{eq:alg2:convergence_alg2}
\frac{1}{t}\sum_{\tau=0}^{t-1}
\|\nabla f(\bar x(\tau))\|^2
\leq \frac{1}{t}\cdot\left[\frac{3.2 \delta(0)}{\eta}
+ \frac{12.8 L^2 R_0 }{1-\rho^2}
+2.4 L^2 d
\sum_{\tau=1}^{\infty}u_\tau^2\right]
\tag{\ref{eq:alg2:convergence_alg2}}
\end{equation}

Now by \eqref{eq:alg2:sum_consensus_error} and \eqref{eq:alg2:convergence_alg2}, we see that
$e_{\mathrm{c}}(t)$ is summable, and
$$
\begin{aligned}
\frac{1}{n}\sum_{\tau=0}^\infty
e_{\mathrm{c}}(\tau)
\leq\ &
\frac{3R_0}{1-\rho^2}
+\frac{4\eta^3L\rho^2(1+2\rho^2)}{(1-\rho^2)^2}
\left[\frac{3.2 \delta(0)}{\eta}
+ \frac{12.8 L^2 R_0 }{1-\rho^2}
+2.4 L^2 d
\sum_{\tau=1}^{\infty}u_\tau^2\right] \\
&+\frac{5 \eta L d \rho^2(1+2\rho^2)}{2(1-\rho^2)^2}
\sum_{\tau=1}^{t-1}u_\tau^2 \\
<\ &
1.6\eta \delta(0)
+\frac{3.2 R_0}{1-\rho^2}
+0.35d\sum_{\tau=1}^\infty u_\tau^2.
\end{aligned}
$$
For the convergence of $s(t)$, we have
$$
\begin{aligned}
&\frac{1}{n}\sum_{\tau=1}^\infty
\|s(\tau)-\mathbf{1}_n\otimes\nabla f(\bar x(\tau-1))\|^2 \\
\leq\ &
\frac{3}{2n}\sum_{\tau=1}^\infty
\|s(\tau)-\mathbf{1}_n\otimes\bar g(\tau)\|^2
+3\sum_{\tau=0}^\infty\|\bar g(\tau)-\nabla f(\bar x(\tau-1))\|^2 \\
\leq\ &
\frac{3}{2n}\sum_{\tau=1}^\infty
e_{\mathrm{g}}(\tau)
+3\sum_{\tau=1}^\infty
\left(\frac{2L^2}{n}
\|x(\tau-1)-\mathbf{1}_n\otimes\bar x(\tau-1)\|^2
+\frac{1}{2}u_\tau^2L^2d
\right) \\
\leq\ &
\left(\frac{3}{2}\cdot\frac{10L}{3\eta}
+6L^2\right)
\left(1.6\eta \delta(0)
+\frac{3.2 R_0}{1-\rho^2}
+0.35d\sum_{\tau=1}^\infty u_\tau^2\right)
+\frac{3}{2}\sum_{\tau=1}^\infty u_\tau^2 L^2 d \\
\leq\ &
9.6L \delta(0)
+ \frac{19.2LR_0}{\eta(1-\rho^2)}
+\frac{2.35}{\eta}
Ld\sum_{\tau=1}^\infty u_\tau^2,
\end{aligned}
$$
where we used \eqref{eq:alg2:temp_bound_1} and $\eta L\leq 1/6$. Finally, since $u_t^2$ is also summable, by Lemma \ref{lemma:alg2:basic_ineq} and the deterministic version of Lemma \ref{lemma:martingale_converge}, we see that $f(\bar x(t))$ converges.
\end{proof}

\subsection{Proof of Theorem~\ref{theorem:alg2_grad_dom}}

We shall still use the notations
$$
x(t)=\begin{bmatrix}
x^1(t) \\
\vdots \\
x^n(t)
\end{bmatrix},
\quad
g(t)=\begin{bmatrix}
g^1(t) \\
\vdots \\
g^n(t)
\end{bmatrix},
\quad
s(t)=\begin{bmatrix}
s^1(t) \\
\vdots \\
s^n(t)
\end{bmatrix},
\qquad
\bar x(t)=\frac{1}{n}\sum_{i=1}^n x^i(t),
\quad
\bar g(t)=\frac{1}{n}\sum_{i=1}^n g^i(t),
$$
and $\delta(t)=f(\bar x(t))-f^\ast$, $e_{\mathrm{c}}(t)=\|x(t)-\mathbf{1}_n\otimes\bar{x}(t)\|^2$, $e_{\mathrm{g}}(t)= \|s(t)-\mathbf{1}_n \otimes \bar g(t)\|^2$. Also recall that the iterations of Algorithm~\ref{alg:multipoint} can be equivalently written as
$$
\begin{aligned}
s(t) &= (W\otimes I_d)(s(t-1)+g(t)-g(t-1)), \\
x(t) &=(W\otimes I_d)(x(t-1)-\eta s(t)),
\end{aligned}
$$
and that
$$
\frac{1}{n}\sum_{i=1}^n s^i(t)=\bar g(t),\qquad\qquad
\bar x(t) = \bar x(t-1)-\eta\bar g(t).
$$

Let $\theta:=\mu/L$. By Lemma \ref{lemma:gradient_upper_bound}, we see that $\mu\leq L$. Notice that the condition on the step size
\begin{equation}\label{eq:alg2:condition_stepsize_grad_dom}
\eta L=\alpha\cdot \left(\frac{\mu}{L}\right)^{\frac{1}{3}}\frac{(1-\rho^2)^2}{14}
\end{equation}
implies $\eta L\leq 1/6$. By \eqref{eq:temp_consensus_alg2} of Lemma~\ref{lemma:alg2_consensus_basic} and Lemma~\ref{lemma:gradient_upper_bound}, we get
$$
\begin{bmatrix}
\frac{5\eta}{2\sqrt{57} L}
e_{\mathrm{g}}(t)\\
e_{\mathrm{c}}(t-1)
\end{bmatrix}
\leq
A
\begin{bmatrix}
\frac{5\eta}{2\sqrt{57} L}
e_{\mathrm{g}}(t-1)\\
e_{\mathrm{c}}(t-2)
\end{bmatrix}+\frac{2n\eta L \rho^2(1+2\rho^2)}{3(1-\rho^2)}
\begin{bmatrix}
4\eta^2 L \delta(t-2)
+\frac{5}{4}u_{t-1}^2d \\
0
\end{bmatrix},
$$
where $A$ is given by \eqref{eq:alg2:mat_A} and the norm of $A$ is given by \eqref{eq:alg2:norm_A}. By solving the inequality $\|A\|\leq 1-(1-\rho^2)^2/21$, we get
$$
(\eta L)^2
\leq
\frac{25(1-\rho^2)^4(13+\rho^2)^2}{\rho^4(
223398 + 411642 \rho^2 + 33642 \rho^4 + 217350 \rho^6 + 18900 \rho^8
)}.
$$
It can be verified that
$$
\frac{25(13+\rho^2)^2}{\rho^4(223398 + 411642 \rho^2 + 33642 \rho^4 + 217350 \rho^6 + 18900 \rho^8)}\geq\frac{1}{14^2}
$$
for all $\rho\in(0,1)$. By the condition \eqref{eq:alg2:condition_stepsize_grad_dom} on $\eta L$, we see that $\|A\|\leq 1-(1-\rho^2)^2/21$. Then, since
$$
\begin{aligned}
\frac{8n \eta^3 L^2\rho^2(1+2\rho^2)}{3(1-\rho^2)}
=\ &
\frac{8n \eta \rho^2(1+2\rho^2)}{3}
\cdot\frac{\alpha^2\theta^{2/3}(1-\rho^2)^3}{196} \\
\leq\ &
\frac{2n\alpha^2\theta^{2/3}\eta}{147} \max_{\rho\in[0,1]}\rho^2(1+2\rho^2)(1-\rho^2)^3
=
\frac{n\alpha^2\theta^{2/3}\eta}{6}
\cdot (1-\chi),
\end{aligned}
$$
where we denote
$$
\chi
:=1-\frac{4}{49}\max_{\rho\in[0,1]}\rho^2(1+2\rho^2)(1-\rho^2)^3\approx 0.9865,
$$
we get
$$
\begin{aligned}
\left\|\begin{bmatrix}
\frac{5\eta}{2\sqrt{57} L}
e_{\mathrm{g}}(t)\\
e_{\mathrm{c}}(t-1)
\end{bmatrix}\right\|
\leq\ &
\left(1-\frac{(1-\rho^2)^2}{21}\right)
\left\|\begin{bmatrix}
\frac{5\eta}{2\sqrt{57} L}
e_{\mathrm{g}}(t-1)\\
e_{\mathrm{c}}(t-2)
\end{bmatrix}\right\| \\
&+
\frac{n\alpha^2\theta^{2/3}\eta }{6}\cdot (1-\chi)\delta(t-2)
+\frac{5n\alpha\theta^{1/3}\rho^2(1+2\rho^2)(1-\rho^2)}{84}u_{t-1}^2 d,
\end{aligned}
$$
where the condition \eqref{eq:alg2:condition_stepsize_grad_dom} was used. Consequently, if we denote
$$
\sigma(t-1)
:=
\frac{2\sqrt{2}L}{n\alpha\theta^{1/3}\sqrt{1-\chi}}
\left\|
\begin{bmatrix}
\frac{5\eta}{2\sqrt{57} L}
e_{\mathrm{g}}(t)\\
e_{\mathrm{g}}(t-1)
\end{bmatrix}
\right\|,
$$
we get
$$
\sigma(t-1)
\leq
\left(1-\frac{(1-\rho^2)^2}{21}\right)
\sigma(t-2)
+
\frac{\sqrt{2}\alpha\theta^{1/3}\sqrt{1-\chi}}{3}
\eta L\cdot \delta(t-2)
+
\frac{5\sqrt{2} \rho^2(1+2\rho^2)(1-\rho^2)}{42\sqrt{1-\chi}}u_{t-1}^2 L d
.
$$
On the other hand, by Lemma \ref{lemma:alg2:basic_ineq} and the fact that $f$ is $\mu$-gradient dominated, we have
$$
\begin{aligned}
\delta(t-1)
\leq\ &
\left(1-\frac{2\eta\mu}{3}\right)\delta(t-2)
+\frac{4\eta L^2}{3n}e_{\mathrm{c}}(t-2)
+\frac{\eta L^2 u_{t-1}^2  d}{3} \\
\leq\ &
\left(1-\frac{2\eta\mu}{3}\right)\delta(t-2)
+\frac{\sqrt{2}\alpha\theta^{1/3}\sqrt{1-\chi}}{3}\eta L
\cdot\sigma(t-2)
+\frac{\eta L^2 u_{t-1}^2 d}{3}.
\end{aligned}
$$
Therefore
\begin{equation}\label{eq:temp_gradient_dom}
\begin{bmatrix}
\sigma(t-1) \\
\delta(t-1)
\end{bmatrix}
\leq
B
\begin{bmatrix}
\sigma(t-2) \\
\delta(t-2)
\end{bmatrix}
+\begin{bmatrix}
\frac{5\sqrt{2}\rho^2(1+2\rho^2)(1-\rho^2)}{14\sqrt{1-\chi}} \\
\eta L
\end{bmatrix} \frac{u_{t-1}^2 Ld}{3},
\end{equation}
where
$$
B:=
\begin{bmatrix}
1-\frac{1}{21}(1-\rho^2)^2
&
\frac{1}{3}\sqrt{2(1-\chi)}\alpha\theta^{1/3}\eta L
\\
\frac{1}{3}\sqrt{2(1-\chi)}\alpha\theta^{1/3}\eta L
& 1-\frac{2}{3}\eta\mu
\end{bmatrix}.
$$
Straightforward calculation shows that
$$
\begin{aligned}
\|B\|
=\ &
1-\frac{(1-\rho^2)^2}{42}
\left(1+\alpha\theta^{4/3}
-\sqrt{(1-\alpha\theta^{4/3})^2
+2(1-\chi)\alpha^4\theta^{4/3}}
\right) \\
\leq\ &
1-\frac{(1-\rho^2)^2}{42}
\left(1+\alpha\theta^{4/3}
-\sqrt{(1-\alpha\theta^{4/3})^2
+2(1-\chi)\alpha\theta^{4/3}}
\right) \\
=\ &
1-\frac{(1-\rho^2)^2}{42}
\left(1+\alpha\theta^{4/3}
-\sqrt{(1-\chi\alpha\theta^{4/3})^2
+(1-\chi^2)\alpha^2\theta^{8/3}}
\right)
\end{aligned}
$$
Since $x\mapsto \sqrt{(1-\chi x)^2+(1-\chi^2)x^2}$ is a convex function over $x\in[0,1]$, it can be shown that
$$
\sqrt{(1-\chi x)^2+(1-\chi^2)x^2}
\leq
1+(\sqrt{2(1-\chi)}-1)x,
$$
and so
$$
\|B\|
\leq 
1-\frac{(1-\rho^2)^2}{42}
\left(2-\sqrt{2(1-\chi)}\right)\alpha\theta^{4/3}
\leq 
1-\frac{(1-\rho^2)^2}{25}\alpha\theta^{4/3},
$$
where we used the fact that $2-\sqrt{2(1-\chi)}>\frac{42}{25}$. By \eqref{eq:temp_gradient_dom}, we then have
$$
\begin{aligned}
\left\|
\begin{bmatrix}
\sigma(t-1) \\
\delta(t-1)
\end{bmatrix}
\right\|
\leq\ &
\left(1-\frac{(1-\rho^2)^2}{25}\alpha\theta^{4/3}\right)
\left\|
\begin{bmatrix}
\sigma(t-2) \\
\delta(t-2)
\end{bmatrix}
\right\|
+\left\|
\begin{bmatrix}
\frac{5\sqrt{2}\rho^2(1+2\rho^2)(1-\rho^2)}{14\sqrt{1-\chi}} \\
\eta L
\end{bmatrix}
\right\|
\frac{u_{t-1}^2 L d}{3} \\
\leq\ &
\left(1-\frac{(1-\rho^2)^2}{25}\alpha\theta^{4/3}\right)
\left\|
\begin{bmatrix}
\sigma(t-2) \\
\delta(t-2)
\end{bmatrix}
\right\|
+5(1-\rho^2) u_{t-1}^2 L d,
\end{aligned}
$$
where we used $\sqrt{1-\chi}>1/9$ and
$$
\left\|
\begin{bmatrix}
\frac{5\sqrt{2}\rho^2(1+2\rho^2)(1-\rho^2)}{14\sqrt{1-\chi}} \\
\eta L
\end{bmatrix}
\right\|
\leq
\left\|
\begin{bmatrix}
\frac{135\sqrt{2}(1-\rho^2)}{14} \\
\frac{1-\rho^2}{14}
\end{bmatrix}
\right\|
<15(1-\rho^2).
$$
By induction we get
$$
\begin{aligned}
\left\|
\begin{bmatrix}
\sigma(t) \\
\delta(t)
\end{bmatrix}
\right\|
\leq\ &
\left(1-\frac{(1-\rho^2)^2}{25}\alpha\theta^{4/3}\right)^t
\left\|
\begin{bmatrix}
\sigma(0) \\
\delta(0)
\end{bmatrix}
\right\|
+5(1-\rho^2) L d
\sum_{\tau=0}^{t-1}
\left(1-\frac{(1-\rho^2)^2}{25}\alpha\theta^{4/3}\right)^{\tau} u^2_{t-\tau},
\end{aligned}
$$
which implies the bound on $f(\bar x(t))-f(x^\ast)$. The bound on $\frac{1}{n}\sum_{i=1}^n\|x^i(t)-\bar x(t)\|^2$ follows from
$$
\frac{n\alpha\theta^{1/3}\sqrt{1-\chi}}{2\sqrt{2}L}
\cdot 5(1-\rho^2)Ld
<
\frac{3n\alpha\theta^{1/3}}{10\sqrt{2}}(1-\rho^2)d
<\frac{3\eta Ld}{1-\rho^2}
$$
as $\sqrt{1-\chi}<3/25$. The bound on $\frac{1}{n}\sum_{i=1}^n \|s^i(t)-\nabla f(\bar x(t-1))\|^2$ follows from
$$
\begin{aligned}
& \frac{1}{n}\|s(t+1)-\mathbf{1}_n\otimes\bar \nabla f(\bar x(t))\|^2 \\
\leq\ &
\frac{3}{2n}\|s(t+1)-\mathbf{1}_n\otimes\bar g(t+1)\|^2
+3\|\bar g(t+1)-\nabla f(\bar x(t))\|^2 \\
\leq\ &
\frac{3}{2n}
e_{\mathrm{g}}(t+1)
+3\left(\frac{2L^2}{n}\|x(t)-\mathbf{1}_n\otimes\bar x(t))\|^2
+\frac{1}{2}u_{t+1}^2 L^2d\right) \\
\leq\ &
\left(\frac{3}{2n}\cdot\frac{10L}{3\eta}
+\frac{6L^2}{n}\right)\cdot \frac{n\alpha\theta^{1/3}\sqrt{1-\chi}}{2\sqrt{2}L}
\left(1-\frac{(1-\rho^2)^2}{25}\alpha\theta^{4/3}\right)^t
\left\|
\begin{bmatrix}
\sigma(0) \\
\delta(0)
\end{bmatrix}
\right\| \\
&+\frac{3}{2}u_{t+1}^2 L^2 d
+\left(\frac{3}{2n}\cdot\frac{10L}{3\eta}
+\frac{6L^2}{n}\right)
\cdot \frac{3\eta Ld}{1-\rho^2}\sum_{\tau=0}^{t-1}
\left(1-\frac{(1-\rho^2)^2}{25}\alpha\theta^{4/3}\right)^{\tau} u_{t-\tau}^2 \\
\leq\ &
\frac{18L}{5(1-\rho^2)^2}\left(1-\frac{(1-\rho^2)^2}{25}\alpha\theta^{4/3}\right)^t
\left\|
\begin{bmatrix}
\sigma(0) \\
\delta(0)
\end{bmatrix}
\right\|
+\frac{18 L^2d}{1-\rho^2}
\sum_{\tau=0}^{t}\left(1-\frac{(1-\rho^2)^2}{25}\alpha\theta^{4/3}\right)^{\tau} u_{t+1-\tau}^2,
\end{aligned}
$$
where we used \eqref{eq:alg2:temp_bound_1} in the second inequality.

\bibliographystyle{IEEEtran}
\bibliography{biblio_revised.bib}

\end{document}